\documentclass[10pt,A4paper]{article}

\RequirePackage{amsthm,amsmath,amsfonts,amssymb}
\RequirePackage[numbers]{natbib}
\RequirePackage[colorlinks,citecolor=blue,urlcolor=blue]{hyperref}
 \RequirePackage[scaled=.9]{helvet}
\usepackage{authblk}
\usepackage{mathtools}
\usepackage[english]{babel}

\usepackage{graphicx}
\usepackage[mathscr]{eucal}
\usepackage{enumerate}
\usepackage{cases}

\usepackage{diagbox}
\numberwithin{equation}{section}
\theoremstyle{plain}
\newtheorem{thm}{Theorem}[section]

\newtheorem{lem}{Lemma}[section]
\newtheorem{prop}{Proposition}[section]

\usepackage{color}

\newcommand{\E}{\mathbb{E}}
\newcommand{\PP}{\mathbb{P}}
\newcommand{\R}{\mathbb{R}}
\newcommand{\F}{\mathcal{F}}

\begin{document}

\title{Volatility and jump activity estimation in a stable Cox-Ingersoll-Ross model}

\author[1]{\textsc{Elise Bayraktar}}
\author[1]{\textsc{Emmanuelle Cl\'ement}}
\affil[1]{LAMA, Univ Gustave Eiffel, Univ Paris Est Creteil, CNRS, F-77447 Marne-la-Vall\'ee, France.}

\date{July, 26 2024}
\maketitle
\noindent
{\bf Abstract.}
We consider the parametric estimation of the volatility and jump activity in a stable Cox-Ingersoll-Ross ($\alpha$-stable CIR) model driven by a standard Brownian Motion and a non-symmetric stable L\'evy process with jump activity $\alpha \in (1,2)$. The main difficulties to obtain rate efficiency in estimating these quantities arise from the superposition of the diffusion component with jumps of infinite variation.  Extending the approach proposed in Mies \cite{MiesMoments},  we address the joint estimation of the volatility, scaling and jump activity parameters  from high-frequency observations of the process and prove that the proposed estimators are rate optimal up to a logarithmic  factor.

\noindent
\textbf{MSC $2020$}.  60G51; 60G52; 60J75; 62F12.

\noindent
\textbf{Key words}: Cox-Ingersoll-Ross model, L\'evy process; Parametric inference; Stable process; Stochastic Differential Equation.

\section{Introduction}
Jump processes are widely studied in the literature due to their many areas of application, especially in financial modeling where interest rates or asset price evolutions are often described by It\^{o} semimartingales. These processes are characterized by the superposition of a drift term, a diffusion part and a jump component and the estimation of the different components is the subject of specific studies depending on the available observations. In this paper we are interested in the estimation of the volatility and the jump activity from high-frequency observations. The estimation of such quantities has concentrated a lot of activity in the last decades and it is well known that the presence of jumps, and more precisely small jumps, makes the estimation of the volatility more difficult. Jump activity is characterized by the Blumenthal-Getoor index $\alpha \in (0,2)$ and when  jumps are of infinite variation, i.e. $\alpha >1$, the usual statistics to estimate the volatility such as the truncated realized quadratic variation have an asymptotic bias. On the other hand, the estimation of the jump activity is disturbed by the presence of a diffusion component since the process is dominated by the Brownian Motion. 

In this work, we focus on processes with infinite variation jumps. Many works have attempted to estimate the volatility or the jump activity in this context and we  briefly present  the most relevant. An efficient estimator of the volatility, with optimal rate of convergence $\sqrt{n}$, has been proposed in Jacod and Todorov \cite{JacodTodorovVolatility} in presence of infinite variation jumps with stable-like behavior. Their method is based on the local  estimation of the volatility by using the real part of the empirical characteristic function and the asymptotic bias is eliminated by combining statistics with different arguments in the characteristic function. Their estimator achieves rate and variance optimality for  $\alpha \in (1,2)$ if the jumps of the driving L\'evy process are symmetric.  For the truncated realized volatility, some debiasing procedures have been studied in Amorino and Gloter \cite{AmorinoGloter} and  in Cooper Boniece and al. \cite{FigueroaLopezVolatility}. By considering different truncation levels and with  a two-step debiasing procedure, efficiency for the estimation of the volatility is obtained in \cite{FigueroaLopezVolatility}, with the restriction $\alpha \in (1,8/5)$. Turning to the estimation of the jump activity, 
A\"{\i}t-Sahalia and Jacod in \cite{AitSahaliaJacodEstimationAlpha} have proposed an estimator of $\alpha$, in presence of a diffusion component, based on jump counting. For locally stable jumps, their estimator is consistent with  rate of convergence $n^{\alpha/10}$ ($n^{\alpha/8}$ in the stable case) but unfortunately  it fails to achieve the optimal rate $n^{\alpha/4} \log(n)^{1-\alpha/4}$ identified in A\"{\i}t-Sahalia and Jacod \cite{AitSahaliaJacodVitesseOptimale}. This estimation method is improved in Bull \cite{Bull} by  combining increments of the process with different time-scales, achieving the near optimal rate of convergence $n^{\alpha/4-\epsilon}$ for $\epsilon$ arbitrarily small. All these methods are non-parametric and apply to semimartingales but are not completely satisfactory. 
In a parametric framework, rate optimality in estimating the jump activity has been attained by Mies  \cite{MiesMoments} for a L\'evy process with characteristic triplet $(\mu, \sigma^2, \nu)$, assuming that the jump part is close to the superposition of $M$ stable processes. Using a method of moments, a joint estimator of the volatility $\sigma^2$, the jump activity indices $\alpha=(\alpha_m)_{1 \leq m \leq M}$ and some scaling  coefficients $r=(r^+_m,r^-_m)_{1 \leq m \leq M}$ is proposed. This procedure leads to an efficient estimator of the volatility (without debiasing procedure and without restriction on the jump activity) and to rate optimal estimators of the jump activity and scaling coefficients. 

In this paper, our aim is to study the joint parametric estimation of the volatility and jump activity of a stochastic differential equation with infinite variation jumps. We will focus our work on the stable Cox-Ingersoll-Ross model (stable CIR model) described by the stochastic equation
$$
  dX_t=(a - b X_t)dt + \sigma X_{t}^{1/2} dB_t + \delta^{1/ \alpha} X_{t-}^{1/ \alpha} dL^{\alpha}_t, 
$$
where $(B_t)$ is a standard Brownian Motion and $(L_t^{\alpha})$ a spectrally positive L\'evy process with jump activity $\alpha \in (1,2)$.
This model  is very popular in financial modeling (see Jiao et al. \cite{JiaoEtAl}, \cite{alpha-Heston})
and also presents some specific difficulties since the coefficients of the equation are not Lipschitz and not bounded away from zero. When $\sigma=0$ (pure-jump process), some estimation results  have been obtained in Li and Ma \cite{Li-Ma}, Yang \cite{Yang} and more recently in Bayraktar and Cl\'ement \cite{BC}. But when $\sigma \neq 0$, there are almost no estimation results in the literature for the stable CIR model and to our knowledge the only existing study concerns the estimation of the drift term $b$ proposed in Barczy et al. \cite{BarczyTransformeeLaplace} from continuous observations of the process. 
Inspired by the work of Mies \cite{MiesMoments}, we propose here a joint estimation method of the volatility coefficient $\sigma^2$, the scaling parameter $\delta$ and the jump activity index $\alpha$ from high-frequency observations of the process that leads to an efficient estimation of the volatility and attains rate optimality up to a factor $\log(n)$ for the estimation of $(\delta, \alpha)$. Our main results are derived in the case $n \Delta_n$ fixed, where $n$ is the number of observations and $\Delta_n$ the step between two consecutive observations (we recall that in this framework and in presence of a Brownian component, the drift parameters $(a,b)$ are not identifiable) but 
also in the long time asymptotics  $n \Delta_n \to + \infty$. 

The paper is organized as follows. In Section \ref{S:setting}, we present the model and notation. In Section \ref{S:main}, we describe the estimation method  and study the asymptotic properties of estimators in both cases $n \Delta_n$ fixed and $n \Delta_n \to + \infty$. Section \ref{S:auxi} contains some intermediate key results that are of intrinsic interest. Proofs are gathered in Section \ref{S:proof}.
%The remaining proofs are given in the supplementary material \cite{PJ-FCIR}.
\section{Setting and notation} \label{S:setting}
We consider $(B_t)_{t\geq0}$ a Brownian motion and $(L^\alpha_t)_{t\geq0}$ an independent stable L\'evy process with exponent $\alpha \in (1,2)$ defined on a filtered probability space $(\Omega, \F, (\F_t)_{t\geq0}, \PP)$, where $(\F_t)_{t\geq0}$ is the natural augmented filtration. We assume that $(L^\alpha_t)_{t\geq0}$ is a spectrally positive L\'evy process with characteristic function
\begin{equation}
    \E \left(e^{iz L_1^{\alpha}} \right) =  \exp \left( - |z|^\alpha \left( 1 - i \tan\left(\frac{\pi \alpha}{2}\right) \text{sgn}(z) \right) \right).
\end{equation}
This non-symmetric L\'evy process is strictly stable (see for example Sato \cite{Sato}) and admits the representation
$$
 L^\alpha_t = \int_0^t \int_{\mathbb{R}_+} z \tilde{N}(ds, dz),  \quad \tilde{N} = N - \overline{N},
$$
where  $N$ is a Poisson random measure with compensator $\overline{N}$ given by
\begin{equation}\label{eq:mesureLevy}
    \overline{N}(dt, dz) = dt F^{\alpha}(dz) \quad \text{for} \quad F^{\alpha}(dz) = \frac{c_\alpha}{z^{\alpha+1}} {\bf 1}_{(0, + \infty)} (z) dz.
\end{equation}
From Lemma 14.11 in Sato \cite{Sato}, we have
\begin{equation} \label{eq:c-alpha}
c_{\alpha}= -\frac{\alpha(\alpha-1)}{\Gamma(2-\alpha) \cos(\pi \alpha/2)}, \mbox{ with } \Gamma(\alpha)= \int_0^{\infty} x^{\alpha-1} e^{-x} dx, \quad \alpha>0.
\end{equation}

We will study the estimation of the parameter $\theta=(\sigma^2, \delta, \alpha)$ of the process (stable CIR model) given by the stochastic differential equation
\begin{equation} \label{eq:EDSCIR}
    dX_t=(a - b X_t)dt + \sigma X_{t}^{1/2} dB_t + \delta^{1/ \alpha} X_{t-}^{1/ \alpha} dL^{\alpha}_t, \quad X_0=x_0 >0.
\end{equation}
Assuming $a >0$ , $b \in \R$, $\delta> 0$, $\sigma > 0$, $x_0 > 0$ and $2a \geq \sigma^2$, we have according to Proposition 3.4. in Jiao et al \cite{JiaoEtAl} that $(X_t)_t$ is positive. 

We use the following notation. 
%The transpose of a matrix $A$ is denoted by $A^T$.
For any random variable $Z$, we write $\E_i(Z) = \E( Z | \F_{i \Delta_n})$ and for $p>0$, $\E^p( Z )=[\E(Z)]^p$ ($\E_i^p(Z)$ is defined accordingly). For any process $Y$, we denote the increments with step size $\Delta_n$ and the symmetrized increments by
\begin{equation*}
 \Delta _i^n Y = Y_{(i+1) \Delta_n} - Y_{i \Delta_n}, \quad \quad \Delta^{s,n}_{i}Y= \Delta^n_{i+1} Y - \Delta^n_{i} Y.
\end{equation*}
%We set
%\begin{equation}
 %   \rho_i^n = \frac{\Delta^n_{i+1} X - \Delta^n_{i} X}{\sqrt{\Delta_n}}
%\end{equation}
%the difference of increments where $\Delta _i^n X = X_{(i+1) \Delta_n} - X_{i \Delta_n}$. We now consider a flattened version of $\overline{\rho}_i^n$, without considering the drift which is negligible
%\begin{equation}
%    \overline{\rho}_i^n = \frac{\sigma X_{i\Delta_n}^{1/2}(\Delta^n_{i+1} W - \Delta^n_i W ) + \delta^{1/\alpha} X_{i\Delta_n}^{1/\alpha}(\Delta^n_{i+1} L^{\alpha} - \Delta^n_i L^{\alpha})}{\sqrt{\Delta_n}},
%\end{equation}
%and separate the Brownian part and the pure-jump part
%\begin{equation}
%    \overline{\rho}_{i}^{1,n} = \frac{\sigma X_{i\Delta_n}^{1/2}(\Delta^n_{i+1} W - \Delta^n_i W )}{\sqrt{\Delta_n}}, \quad 
 %   \overline{\rho}_{i}^{2,n} = \frac{\delta^{1/\alpha} X_{i-\Delta_n}^{1/\alpha}(\Delta^n_{i+1} L^{\alpha} - \Delta^n_i L^{\alpha}))}{\sqrt{\Delta_n}}.
%\end{equation}
We write $\Longrightarrow$ for the convergence in distribution, and $\xRightarrow[]{\mathcal{L}-s}$ for the stable convergence in law.
Throughout the paper, we denote by C all positive constants (or $C_p$ if it depends on an auxiliary parameter $p$), which might change from line to line.

\section{Main results} \label{S:main}
From now on, we assume that we observe $(X_{i \Delta_n})_{0 \leq i \leq n}$ that solves \eqref{eq:EDSCIR} for the parameter values $(a_0,b_0, \theta_0)$ satisfying assumption {\bf H} below.

\noindent
{\bf H. } We assume that $\theta_0=(\sigma_0^2, \delta_0, \alpha_0) \in \Theta= (0, + \infty) \times (0, + \infty) \times (1,2)$
and $a_0 > \sigma_0^2$. 

Under {\bf H}, $(X_t)_{t \geq 0}$ is  positive as previously mentioned, but also satisfies that $\E(1/ X_t^2) < \infty$ (see Proposition \ref{Th:Moments}).  To estimate $\theta_0$, we propose a moment method based on the increments of $X$, that extends to stochastic equations the estimation method proposed in Mies \cite{MiesMoments} for L\'evy processes.   
For $f: \R \mapsto \R^3$, we consider the estimating function $F_n(\theta)$ defined by
\begin{equation} \label{E:estf}
F_n(\theta) = \frac{1}{n} \sum_{i=0}^{\lfloor n/2 \rfloor - 1} \left[ f\left(\frac{u_n}{\sqrt{\Delta_n}} \frac{\Delta^{s,n}_{2i} X}{\sqrt{X_{2i \Delta_n}}}\right) 
- P_n f(X_{2i \Delta_n}, \theta)\right],
\end{equation}
where the centering term $P_n f(x ,\theta)$, defined by \eqref{eq:defPnf},  is constructed on the following way. Considering  the Euler approximation of $\Delta^{s,n}_{2i} X$ omitting the drift term
 \begin{equation*}
  \Delta^{s,n}_{2i}  \overline{X}  = \sigma_0 X_{2i\Delta_n}^{1/2}\Delta^{s,n}_{2i} B  + \delta_0^{1/\alpha_0} X_{2i\Delta_n}^{1/\alpha_0}\Delta^{s,n}_{2i} L^{\alpha_0},
\end{equation*}
we approximate $\E_{2i}f\left(\frac{u_n}{\sqrt{\Delta_n}} \frac{\Delta^{s,n}_{2i} X}{\sqrt{X_{2i \Delta_n}}}\right) $ by $\E_{2i}f\left( \frac{u_n}{\sqrt{\Delta_n}} \frac{\Delta^{s,n}_{2i} \overline{X}}{\sqrt{X_{2i \Delta_n}}}\right)$. Now observing that $(\Delta^{s,n}_{2i} B , \Delta^{s,n}_{2i} L^{\alpha})$ has the distribution 
of $(\sqrt{2} \sqrt{\Delta_n} B_1, (2\Delta_n)^{1/ \alpha} S_1^{\alpha})$, where $S_1^{\alpha}$ is a symmetric stable random variable, independent of $B_1$, with characteristic function $\E e^{i z S_1^{\alpha}}=e^{-|z|^{\alpha}}$, we have
%$$
%\E_{2i}f\left( \frac{u_n}{\sqrt{\Delta_n}} \frac{\Delta^{s,n}_{2i} \overline{X}}{\sqrt{X_{2i \Delta_n}}}\right)=P_nf( X_{2i \Delta_n}, \theta_0),
%$$
%with
\begin{align} \label{eq:defPnf}
P_n f(x, \theta)  = 
\E_{2i}f\left( \frac{u_n}{\sqrt{\Delta_n}} \frac{\Delta^{s,n}_{2i} \overline{X}}{\sqrt{X_{2i \Delta_n}}}\right)=
  \E \left[ f \left(\sqrt{2} u_n \sigma B_1 + \frac{(2 \delta x)^{1/\alpha}}{\sqrt{x}} u_n \Delta_n^{1/\alpha - 1/2} S^{\alpha}_1 \right) \right].
\end{align}
At this stage, we assume that $u_n$ goes to zero as $n$ goes to infinity and $u_n / \sqrt{\Delta_n} $ goes to infinity, we will precise later the rate of convergence of the sequence $(u_n)$.  
We estimate the parameter $\theta_0$  by solving the estimating equation $F_n(\theta)=0$ for $F_n$ defined by \eqref{E:estf} with $f=(f_k)_{1 \leq k \leq 3}$, $f_1(x)=\cos(x)$, $f_2(x)=1- K(x)$, $f_3(x)=f_2(2x)$ and
\begin{equation*} \label{eq:defK}
K(x)=
\left\{
\begin{array}{l}
1 \; \text{ if} \; |x| \leq 1, \\ 
(1+ \exp(\frac{1}{2-|x|}-\frac{1}{1-|x|}) )^{-1} \; \text{ if} \; 1 < |x| < 2 ,\\
0 \; \text{ if} \; |x| \geq 2. 
\end{array} 
\right.
\end{equation*}
The choice of $f_1$ is natural to estimate $\sigma^2$ and has been used in Jacod and Todorov \cite{JacodTodorovVolatility}. For this choice,  
the expression of $P_n f_1$ is explicit, it is the real part of the characteristic function of $u_n \Delta^{s,n}_{2i}\overline{X}/\sqrt{\Delta_n X_{2i \Delta_n}}$
\begin{equation} \label{eq:Pncos}
P_n f_1(x, \theta)= \exp(-u_n^2 \sigma^2-2 \delta x^{1- \alpha/2} u_n^{\alpha} \Delta_n^{1- \alpha/2}).
\end{equation}
To estimate $\delta$ and $\alpha$, related to the jump component of the process, we consider here  smooth  truncations of the small jumps, with different truncation levels.
$P_n f_2$ and  $P_n f_3$ are not explicit (an asymptotic expansion is given in  Lemma \ref{L:expan}) but easy to compute numerically. 

The asymptotic properties of estimators depend on the connections between the number of observations $n$ and the discretization step $\Delta_n$ and we study separately the two cases $n \Delta_n$ fixed and $n \Delta_n \rightarrow + \infty$.

\subsection{Fixed observation window}

We consider that $n \Delta_n$ is fixed and without loss of generality, we assume that $n \Delta_n =1$.
We first prove the existence of a joint estimator with non-diagonal rate of convergence $\Lambda_n$  defined by
\begin{equation} \label{E:ratelambda}
 \Lambda_n(\theta) = \frac{1}{\sqrt{n}} \left(\begin{array}{ccc}
1 & 0 & 0 \\
0 & \frac{1}{ u_n^{\alpha/2} \Delta_n^{1/2 - \alpha/4}} & -\delta \ln(\frac{u_n}{\sqrt{\Delta_n}}) \frac{1}{ u_n^{\alpha/2} \Delta_n^{1/2 - \alpha/4}} \\
0 & 0 & \frac{1}{ u_n^{\alpha/2} \Delta_n^{1/2 - \alpha/4}} \end{array}\right).
\end{equation}
\begin{thm}\label{Th:Existence} 
We assume {\bf H},  $n\Delta_n = 1$ and  $u_n=1/[\ln(1/ \Delta_n)]^{p}$ with $p>1/2$. Then there exists a sequence of random vectors $(\hat{\theta}_n)$ solving $F_n(\hat{\theta}_n) = 0$ with probability tending to one, i.e. $\PP(F_n(\hat{\theta}_n) = 0) \to 1$.
This sequence satisfies $\hat{\theta}_n \to \theta_0$ in probability as $n \to \infty$, and
\[ \Lambda_n(\theta_0)^{-1} (\hat{\theta}_n - \theta_0)  \xRightarrow[n \to \infty]{\mathcal{L}-s} W(\theta_0)^{-1} \Sigma(\theta_0)^{1/2} \mathcal{N},\]
where $\mathcal{N}$ is a standard Gaussian variable independent of $\mathcal{F}_1$.  $\Sigma$ is the symmetric definite positive matrix and $W$ is the non-singular matrix respectively defined   by
\begin{equation} \label{E:sigma}
 \Sigma(\theta) =\frac{1}{2} \left(\begin{array}{ccc}
 2\sigma^4 & 0 & 0 \\
0 & c_{\alpha} \delta I(\alpha,X) \int_{\R} \frac{f_2^2 (v)}{|v|^{\alpha+1}}dv & c_{\alpha} \delta I(\alpha,X) \int_{\R} \frac{(f_2 f_3) (v)}{|v|^{\alpha+1}} dv 
\\
0 &c_{\alpha} \delta I(\alpha,X)\int_{\R} \frac{(f_2 f_3 )(v)}{|v|^{\alpha+1}} dv & c_{\alpha} \delta I(\alpha,X) \int_{\R} \frac{f_3^2 (v)}{|v|^{\alpha+1}} dv
\end{array}\right), 
\end{equation}
\begin{equation} \label{E:wtheta}
W(\theta) = \frac{1}{2}\left(\begin{array}{ccc}
1 & 0 & 0 \\
0 & -\psi(\alpha) I(\alpha,X)
& -\delta \partial_{\alpha}[ \psi(\alpha)I(\alpha,X)] 
\\
0 & - 2^{\alpha} \psi(\alpha) I(\alpha,X)
& -\delta \partial_{\alpha}[ 2^{\alpha}\psi(\alpha)I(\alpha,X)]
\end{array}\right),
\end{equation}
with  $I(\alpha,X)= \int_0^1 X_t^{1-\alpha/2}dt$,  $\psi(\alpha)=c_{\alpha} \int \frac{f_2(v)}{|v|^{\alpha+1}} dv$ and $c_{\alpha}$ given by \eqref{eq:c-alpha}. 
%\begin{align*}
%W_{2,3}(\theta)  &=-\delta \left(\partial_{\alpha}\psi(\alpha)\int_0^1 X_t^{1-\alpha/2} dt-\psi(\alpha) \int_0^1 X_t^{1-\alpha/2}\ln(\sqrt{X_t}) dt   \right) , \\
%W_{3,3}(\theta) & =- \delta\left(\partial_{\alpha}(2^{\alpha}\psi(\alpha))\int_0^1 X_t^{1-\alpha/2} dt -2^{\alpha}\psi(\alpha) \int_0^1 X_t^{1-\alpha/2}\ln(\sqrt{X_t} ) dt   \right).
%\end{align*}
\end{thm}

We deduce from Theorem \ref{Th:Existence} that $\sqrt{n}(\hat{\sigma}^2_n - \sigma_0^2) \to \mathcal{N}(0, 4 \sigma_0^4)$, so the estimator is rate optimal and achieves  the optimal variance up to a factor 2, which is explained by the use of the symmetrized increments  $\Delta^{s,n}_{2i} X$. Optimality for the marginal estimation of $(\delta, \alpha)$ has been discussed in A\"it-Sahalia and Jacod \cite{AitSahaliaJacodVitesseOptimale}. For the process $X_t= \sigma B_t + \delta L_T^{\alpha}$ with $\sigma>0$, they show that the optimal  rates of convergence  for the estimation of $\delta$ and $\alpha$  are respectively $
1/ ((\log(1/ \Delta_n)^{\alpha/4} \Delta_n^{\alpha/4})$ and $\log(1/ \Delta_n)/ ((\log(1/ \Delta_n)^{\alpha/4}\Delta_n^{\alpha/4})$. But from Mies \cite{MiesMoments} a diagonal normalisation with marginal optimality for each parameter leads to a singular information matrix, suggesting that the optimal rate of convergence in the joint estimation of $(\delta, \alpha)$ is non-diagonal with slower rate of convergence by a factor $\log(1/\Delta_n)$ and probably achieved by the choice $u_n=1/(\log(1/ \Delta_n))^{1/2}$ in \eqref{E:ratelambda}. This choice is theoretically possible up to a constant depending on the true value of the parameters since the proof of Lemma \ref{L:expan} still works with $u_n=1/(4 \sigma_0(\log(1/ \Delta_n))^{1/2})$ (see \eqref{eq:BI1}). This choice is not possible in practice, this is why we assume the near optimal condition $u_n=1/(\log(1/ \Delta_n))^p$ with $p>1/2$ in Theorem \ref{Th:Existence}. 

We now assume that $\alpha_0$ is known and only consider the estimation of $\beta=(\sigma^2, \delta)$ by solving $F_n(\sigma^2, \delta, \alpha_0)=0$, where $F_n$ is defined by \eqref{E:estf} with $f=(f_1,f_2)$. In that case,
we can also prove a uniqueness result, which implies that any solution to the estimating equation is consistent. 
\begin{thm}\label{Th:Unicité}
We assume that $\alpha_0$ is known and that $(\sigma_0^2, \delta_0)$ belongs to the interior of  a compact subset of $(0, + \infty) \times (0, + \infty)$. Then under the assumptions of Theorem \ref{Th:Existence}  any sequence $(\hat{\sigma}^2_n, \hat{\delta}_n)$ that solves $ F_n(\sigma^2, \delta,\alpha_0) = 0$ converges in probability to $(\sigma_0^2, \delta_0)$. Such a sequence exists and is unique in the sense that if $(\tilde{\sigma}^{2}_n, \tilde{\delta}_n)$ is another sequence that solves $ F_n(\sigma^2, \delta,\alpha_0) = 0$ then $\PP( (\hat{\sigma}^2_n, \hat{\delta}_n) \neq (\tilde{\sigma}^{2}_n, \tilde{\delta}_n))$ goes to zero as n goes to infinity. This sequence obviously satisfies the asymptotic properties of Theorem \ref{Th:Existence}.
%With\[ W(\theta_0) = \left(\begin{array}{cc}
%\frac{1}{2} & 0 \\
%0 & -\frac{\psi(\alpha_0)}{2} \int_0^1 X_t^{1-\alpha_0/2} dt\end{array}\right)^{-1}
% \left(\begin{array}{cc}
% \sigma^4_0 & 0\\
%0 & \frac{I(\theta_0,X)}{2} \int_{\R} \frac{f_2^2 (v)}{|v|^{\alpha_0+1}}dv \end{array}\right)^{1/2}\]
%we have the stable convergence in law
%\[ \left(\begin{array}{cc}
%\sqrt{n} & 0  \\
%0 & \sqrt{n} \Delta_n^{1/2 - \alpha_0/4} u_n^{\alpha_0/2}\end{array}\right) \left(\begin{array}{cc}
%\hat{\sigma}^2_n - \sigma_0^2  \\
%\hat{\delta}_n - \delta_0 \end{array}\right) \xrightarrow[n \to \infty]{\mathcal{L}-s} W(\theta_0) \mathcal{N}.\]
\end{thm}

\subsection{Ergodic case}

We consider in this section the long time asymptotics $n \Delta_n \to + \infty$ and we restrict our study to the  case $b_0 > 0$. With this assumption, the process $(X_t)_{t \geq 0}$ is geometrically ergodic (we refer to Li and Ma \cite{Li-Ma} Proposition 2.2 and Theorem 2.5) and the stationary distribution $\pi_{0}$ has Laplace transform
\[ \int_0^{+ \infty} e^{-\lambda y} \pi_{0}(dy) = \exp \left( - \int_0^\lambda \frac{a_0 x}{\frac{\sigma_0^2}{2} x^2 + \frac{\overline{\delta_0}^{\alpha_0}}{\alpha_0} + b_0 x} dx \right) ,\]
with $\overline{\delta_0}^{\alpha_0}=\delta_0 \alpha_0/| \cos(\frac{\pi \alpha_0}{2}) |$.
Since the Laplace transform is $\mathcal{C}^1$ at zero,  we check easily that $\int_0^{\infty} x \pi_{0}(dx) < + \infty$,  and setting $\overline{I}(\alpha)=\int_0^{\infty} x^{1-\alpha/2} \pi_{0}(dx)$ we deduce that $\alpha \mapsto \overline{I}(\alpha)$ is derivable for $\alpha \in (1,2)$.

\begin{thm}\label{Th:ExistenceTempslong} 
We assume {\bf H} and $b_0>0$,  $n \Delta_n \to + \infty$ ,  $u_n=1/[\ln(1/ \Delta_n)]^{p}$ with $p>1/2$, and the additional condition $\sqrt{n} \Delta_n^{1/\alpha_0 + \alpha_0/4 - 1/2-\epsilon_0} \to 0$ for $\epsilon_0>0$. Then the result of Theorem \ref{Th:Existence} holds  with
\[ \Lambda_n(\theta_0)^{-1} (\hat{\theta}_n - \theta_0)  \xRightarrow[n \to \infty]{} \overline{W}(\theta_0)^{-1} \overline{\Sigma}(\theta_0)^{1/2} \mathcal{N},\]
where $\mathcal{N}$ is a standard Gaussian variable.  $\overline{\Sigma}$ is the symmetric definite positive (deterministic) matrix  and $\overline{W}$ is the (deterministic) non-singular matrix respectively defined  by \eqref{E:sigma} and \eqref{E:wtheta} replacing  $I(\alpha,X)$ by $\overline{I}(\alpha)= \int_0^{+ \infty} x^{1-\alpha/2} \pi_{0}(dx)$.
\end{thm}

The condition $\sqrt{n} \Delta_n^{1/\alpha_0 + \alpha_0/4 - 1/2-\epsilon_0} \to 0$ is compatible with $n \Delta \to + \infty$ if $\alpha_0 \in (1,2)$ but stronger for $\alpha_0$ near $2$.
We also mention that we can also prove in the ergodic case a uniqueness result for the estimation of $(\sigma^2, \delta)$ (similar to Theorem \ref{Th:Unicité}) if we assume that $\alpha_0$ is known.

%%%%%%%%%

%\begin{thm}\label{Th:UnicitéL}
%We assume that $\alpha_0$ is known. We assume that $(\sigma_0^2, \delta_0)$ belongs to the interior of A, a compact subset of $(0, + \infty) \times \R$.\\
%2. We assume that $b_0 > 0$. We assume that $n \Delta_n \to + \infty$ with $\sqrt{n} \Delta_n^{1/\alpha_0 + \alpha_0/4 - 1/2} \to 0$ and $u_n=1/[\ln(1/ \Delta_n)]^{p}$ with $p>1/2$. Then any sequence $(\hat{\sigma}^2_n, \hat{\delta}_n)$ that solves $ \overline{F}_n(\sigma^2, \delta) = 0$ converges in probability to $(\sigma_0^2, \delta_0)$. Such a sequence exists and is unique in the sense that if $(\hat{\sigma}^{2'}_n, \hat{\delta}'_n)$ is another sequence that solves $ \overline{F}_n(\sigma^2, \delta) = 0$ then $\PP( (\hat{\sigma}^2_n, \hat{\delta}_n) \neq (\hat{\sigma}^{2'}_n, \hat{\delta}'_n))$ goes to zero as n goes to infinity. With
%\[ W'(\theta_0) = \left(\begin{array}{cc}
%\frac{1}{2} & 0 \\
%0 & - \frac{\psi(\alpha_0)}{2} \int_0^{+ \infty} x^{1-\alpha_0/2} \pi(dx) \end{array}\right)^{-1}
% \left(\begin{array}{cc}
% \sigma^4_0 & 0\\
%0 & \frac{I'(\theta_0,X)}{2} \int_{\R} \frac{f_2^2 (v)}{|v|^{\alpha_0+1}}dv \end{array}\right)^{1/2}\]
%we have the stable convergence in law
%\[ \left(\begin{array}{cc}
%\sqrt{n} & 0  \\
%0 & \sqrt{n} \Delta_n^{1/2 - \alpha_0/4} u_n^{\alpha_0/2} \end{array}\right) \left(\begin{array}{cc}
%\hat{\sigma}^2_n - \sigma_0^2  \\
%\hat{\delta}_n - \delta_0 \end{array}\right) \xrightarrow[n \to \infty]{\mathcal{L}-s} W'(\theta_0) \mathcal{N}.\]
%\end{thm}

\section{Auxiliary results} \label{S:auxi}

%We give in this section the key tools for the proof of the main results. 

\subsection{Moments and approximation} \label{Ss:moment}

Since $X_t$ is  positive but not lower bounded by some strictly positive constant, we first establish conditions to ensure  that  $\E(1/ X^p_t)$ is finite. We also give some bounds for the moments of $(X_t)_t$ and its increments.

 \begin{prop}\label{Th:Moments} 
Let $(X_t)_{t \geq 0}$ be the solution of \eqref{eq:EDSCIR} with $a >0$ , $b \in \R$, $\delta> 0$, $\sigma > 0$, $x_0 > 0$ and $2a > \sigma^2$. Then we have 
$$
\begin{array}{l}
(i) \quad \text{for}  \; 1 \leq p < \frac{2a}{\sigma^2}, \quad \sup_{t \geq 0} \E(\frac{1}{X_t^p}) < + \infty  , 
\end{array}
$$
and for $ 1 \leq p < \alpha$, $\exists C_p >0$ such that $\forall t_1,t_2 \in [0,1]$ and $\forall i \in [ \! [ 0; n-1 ] \! ]$
$$
\begin{array}{l}
(ii) \;   \E \left( \sup_{t_1 \leq u \leq t_2} X_u^p  |\mathcal{F}_{t_1}\right) \leq C_p(1+ X_{t_1}^p), \\ 
(iii)  \; 
\E_i \left( \sup_{s\in [i \Delta_n, (i+1)\Delta_n]} \left|  \int_{i \Delta_n}^s X_{t-}^{1/ \alpha} dL^\alpha_t  \right|^{p} \right) \leq C_p \Delta_n^{p/\alpha} \left( 1 + X_{i \Delta_n}^{p / \alpha} \right), \\
 (iv) \;
        \E_i \left( \sup_{s\in [i \Delta_n, (i+1)\Delta_n]} \left|X_s - X_{i \Delta_n}\right|^{p}\right) \leq C_p\Delta_n^{p/2}\left( 1 + X_{i \Delta_n}^{p} \right), 
\\
(v) \;
 \E_i \left( \left| \int_{i \Delta_n}^{(i+1) \Delta_n} (X_{s-}^{1/ \alpha} - X_{i \Delta_n}^{1/ \alpha}) dL^\alpha_s  \right|^{p} \right) \leq C_p \Delta_n^{p/\alpha+p/2} \left( 1 + \frac{1}{X_{i \Delta_n}^{p(1-1/ \alpha)}} + X_{i \Delta_n}^{p/ \alpha} \right).

\end{array}
$$
\end{prop}

We next give a sharp bound for the Euler approximation (without drift) of the symmetrized normalized increments of $X$. This result can be seen as an extension of Lemma 11  in Jacod Todorov \cite{JacodTodorovVolatility}. We obtain here a better rate of weak approximation and this result is a key step in the proof of Proposition \ref{Th:TCL}.
To shorten the notation, we set
 \begin{equation} \label{eq:rho}
      \rho_i^{n} = \frac{u_n}{\sqrt{\Delta_n}} \frac{\Delta^{s,n}_{i} X}{\sqrt{X_{i \Delta_n}}} \quad \text{and} \quad
 \overline{\rho}_i^{n} = \frac{u_n}{\sqrt{\Delta_n}} \frac{\Delta^{s,n}_{i} \overline{X}}{\sqrt{X_{i \Delta_n}}},
 \end{equation}
 where
$
  \Delta^{s,n}_{i}  \overline{X}  = \sigma_0 X_{i\Delta_n}^{1/2}\Delta^{s,n}_{i} B  + \delta_0^{1/\alpha_0} X_{i\Delta_n}^{1/\alpha_0}\Delta^{s,n}_{i} L^{\alpha_0}
$
and $(X_t)_{t \geq 0}$ solves \eqref{eq:EDSCIR} with $a_0 >0$ , $b_0 \in \R$, $\delta_0> 0$, $\sigma_0 > 0$, $x_0 > 0$ and $2a_0 \geq \sigma_0^2$.
%%%%%%
\begin{lem}\label{L:ramenerapplati}
 Let $f$ be an even bounded $\mathcal{C}^2$ function with bounded derivatives. Then we have (with $\rho_i^n$ and $\overline{\rho}_i^{n}$ defined by \eqref{eq:rho}) for $\epsilon>0$, $\exists C_{\epsilon}$ such that $\forall i \in [ \! [ 0; n-1 ] \! ]$
    \begin{align}
  |\E_{i}  (f (\rho_i^n) - f( \overline{\rho}_i^{n}))|  \leq  & C _{\epsilon} u_n \Delta_n^{1/ \alpha_0- \epsilon} \left(1+X_{i\Delta_n}^{1/2} 
  %\right. \nonumber\\
%& \left. +
  \frac{1}{X_{i\Delta_n}}\left[1 + \frac{1}{\Delta_n}\E_{i} \int_{i\Delta_n}^{(i+2)\Delta_n} \frac{dv}{X_v} \right]\right), \label{eq:cos-cos}
\end{align}
and for $1 \leq p< \alpha_0$,  $\exists C_p$ such that $\forall i \in [ \! [ 0; n-1 ] \! ]$
 \begin{align}  
    \E_{i}  (| f ( \rho_i^{n})  - f( \overline{\rho}_i^{n})|^p ) \leq C_p u_n^p \Delta_n^{p /2}(1+ \frac{1}{X_{i\Delta_n}^{p}} +X_{i\Delta_n}^{p/2}). \label{eq:cos-cosQ}    
     \end{align}
\end{lem}

%\begin{rem} This result is similar to Lemma 11 in Jacod Todorov \cite{JacodTodorovVolatility}. Following the proof of Lemma 11 leads to a rate $\Delta_n^{1 - \alpha_0 - \epsilon}$, rate which we improve in this Lemma.
%\end{rem}

Our estimating equations are based on the computation of the expectation $P_nf(x, \theta)$ given by \eqref{eq:defPnf}. For  $f(x) =\cos(x)$ this expectation has an explicit expression but not for $f(x)=1 -K(x)$. The next lemma gives an asymptotic expansion 
of $P_nf(x, \theta)$ for functions vanishing in a neighborhood of zero. It is similar to Lemma 6.1 (i) in \cite{MiesMoments} derived from It\^{o}'s Formula.  Here we propose a different proof, which easily allows higher order expansion,  by using the explicit form of the conditional distribution of the Euler approximation $ \Delta^{s,n}_{i}  \overline{X} $ as a convolution between a Gaussian distribution and a Stable distribution.

\begin{lem}\label{L:expan}
Let $f:\R \mapsto \R$ be a bounded even function such that $f=0$ on $[-\eta,+ \eta]$ for some $\eta>0$. We assume that $u_n [\ln(1/ \Delta_n)]^{1/2} \rightarrow 0$. Then we have $\forall x>0$
\begin{eqnarray*}
P_n f(x, \theta)= u_n^{\alpha} \Delta_n^{1-\alpha/2}c_{\alpha} \delta x^{1-\alpha/2} \int \frac{f(v)}{|v|^{\alpha+1}} dv 
 +u_n^{\alpha} \Delta_n^{1-\alpha/2} \varepsilon_n(\theta,x), \quad \text{with $c_{\alpha}$  given by \eqref{eq:c-alpha},}
\end{eqnarray*}
where   for any compact subset $A$ of $\Theta$, 
$
\sup_{\theta \in A} |\varepsilon_n(\theta,x)| \leq (1+x) \varepsilon_n,  
$
with $\varepsilon_n \rightarrow 0$.
\end{lem}

%%%
\subsection{Convergences} \label{Ss:conv}
 
 The asymptotic behavior of the estimator $\hat{\theta}_n$ is derived using the theory of estimating equations (see \cite{JacodSorensenUniforme}) and is based on the
 convergence  of  $F_n(\theta_0)$ defined by \eqref{E:estf}  and  its gradient $\nabla_{\theta} F_n(\theta)$.
We set 
\begin{equation} \label{E:An}
%A_n (\theta) = \sqrt{n} \left(\begin{array}{ccc}
%\frac{1}{u_n^2} & 0 & 0 \\
%0 & \frac{1}{u_n^{\alpha/2} \Delta_n^{1/2 - \alpha/4}} & 0  \\
%0 & 0 & \frac{1}{u_n^{\alpha/2}\Delta_n^{1/2 - \alpha/4}} \end{array}\right).
A_n (\theta) = \sqrt{n} \; diag\left(\frac{1}{u_n^2} ,\frac{1}{u_n^{\alpha/2} \Delta_n^{1/2 - \alpha/4}}  , \frac{1}{u_n^{\alpha/2} \Delta_n^{1/2 - \alpha/4}} \right). 
\end{equation}
We first state a Central Limit Theorem for $F_n(\theta_0)$ with rate $A_n(\theta_0)$.
\begin{prop}\label{Th:TCL} 
We assume {\bf H} and $u_n=1/[\ln(1/ \Delta_n)]^{p}$ with $p>1/2$. 
\begin{enumerate}
\item
If $n \Delta_n = 1$, then we have the stable convergence in law
\[ A_n(\theta_0) F_n(\theta_0) \xRightarrow[n \to \infty]{\mathcal{L}-s} \Sigma(\theta_0)^{1/2} \mathcal{N},\]
where $\mathcal{N}$ is a standard Gaussian variable independent of $\mathcal{F}_1$ and $\Sigma(\theta_0)$ is the symmetric definite positive matrix defined by \eqref{E:sigma}.
\item
 If $n \Delta_n \to + \infty$, $b_0>0$ and   $\sqrt{n} \Delta_n^{1/\alpha_0 + \alpha_0/4 - 1/2-\epsilon_0} \to 0$  for some $\epsilon_0 >0$, then we have the  convergence in distribution
\[ A_n(\theta_0) F_n(\theta_0) \xRightarrow[n \to \infty]{} \overline{\Sigma}(\theta_0)^{1/2} \mathcal{N},\]
where $\mathcal{N}$ is a standard Gaussian variable and $\overline{\Sigma}(\theta_0)$ is the symmetric definite positive matrix defined by \eqref{E:sigma} replacing  $I(\alpha,X)$ by $\overline{I}(\alpha)= \int_0^{+ \infty} x^{1-\alpha/2} \pi_{0}(dx)$.
\end{enumerate}
\end{prop}

Turning to the convergence of $\nabla_{\theta} F_n$ and recalling that $\Lambda_n(\theta)$ is defined by \eqref{E:ratelambda}, we obtain the following convergence.

\begin{prop}\label{Th:LFGN} Let $A$ be a compact subset of $\Theta$. We assume {\bf H} and $u_n=1/[\ln(1/ \Delta_n)]^{p}$ with $p>1/2$.
\begin{enumerate}
\item
 If $n \Delta_n = 1$, then we have the  convergence in probability
\[ \sup_{\theta \in A} |A_n(\theta) \nabla_\theta F_n(\theta) \Lambda_n(\theta) - W(\theta)| \to 0, \]
where $W(\theta)$ is the non-singular matrix defined by \eqref{E:wtheta}.
\item
If  $n \Delta_n \to + \infty$ and $b_0>0$, then we have the  convergence in probability
\[ \sup_{\theta \in A} |A_n(\theta) \nabla_\theta F_n(\theta) \Lambda_n(\theta) - \overline{W}(\theta)| \to 0, \]
where $\overline{W}(\theta)$ is  defined by \eqref{E:wtheta} replacing  $I(\alpha,X)$ by $\overline{I}(\alpha)= \int_0^{+ \infty} x^{1-\alpha/2} \pi_{0}(dx)$.
\end{enumerate}
\end{prop}

\section{Proofs} \label{S:proof}
\subsection{Auxiliary results}
\subsubsection{Proof of Proposition \ref{Th:Moments}}

We only prove $(i)$ since the other items can be deduced with similar arguments as the ones given for the proof of Proposition 2.2 in \cite{BC}.
%We have
%\[ \E(\frac{1}{X_t^p}) =  C_p \int_{\mathbb{R}_+} u^{p-1} \E(e^{-u X_t}) du.\]
The Laplace transform  of $X_t$ has the explicit expression (we refer to Li \cite{Li}, see also Section 3.1 in \cite{JiaoEtAl} and Proposition 2.1. in \cite{BarczyTransformeeLaplace})
\begin{equation} \label{E:laplace}
\E(e^{-u X_t}) = \exp \left(-x_0 v_t(u) - \int_{v_t(u)}^u \frac{F(z)}{R(z)} dz \right),
\end{equation}
where $t \to v_t(u)$ is the unique locally bounded solution of 
\begin{equation} \label{eq:Laplace Eq Diff}
    \frac{\partial}{\partial t} v_t(u) = - R(v_t(u)), \quad v_0(u)=u.
\end{equation}
For the $\alpha$-CIR model, the branching mechanism and the immigration rate are given by
\[R(z)= \frac{\sigma^2}{2} z^2 + \frac{\overline{\delta}^\alpha}{\alpha} z^\alpha +bz, \quad F(z)=az, \quad z \in [0, + \infty),\]
with $\overline{\delta}=\delta^{1/ \alpha} (\alpha/| \cos(\frac{\pi \alpha}{2}) |)^{1/ \alpha}$. Setting $z_0 = \inf \{ z>0, \; R(z) \geq 0 \}$, $z_0 = 0$ when $b \geq 0$ and $z_0 > 0$ when $b < 0$ with $z_0 < + \infty$ as $\lim_{z \to + \infty} R(z) = + \infty$. From Corollary 3.2 of Li \cite{Li}, $\lim_{t \to \infty} v_t(u) = z_0$ and $\forall \; u > z_0$, $t \to v_t(u)$ is decreasing. 
Now we have $\forall t \geq 0$
\[ \E(\frac{1}{X_t^p}) = C_p \int_{\mathbb{R}_+} u^{p-1} \E(e^{-u X_t}) du\leq C_p(z_0) +  C_p  \int_{z_0+1}^{+ \infty} u^{p-1} \E(e^{-u X_t}) du.\]
We fix $\epsilon \in (0, 1)$ and for $u \geq z_0 + 1$ we set $t_u^\epsilon = \inf \{ t\geq 0, v_t(u) \leq u^\epsilon \}$. With $u_{0, \epsilon} = \max( z_0 + 1,z_0^{1/\epsilon} )$
\begin{align}
    \sup_{t \geq 0} \; \int_{u_{0, \epsilon}}^{+ \infty} u^{p-1} \E(e^{-u X_t}) du \leq \int_{u_{0, \epsilon}}^{+ \infty} u^{p-1} \left[ \sup_{t \in [0, t_u^\epsilon]} \;\E(e^{-u X_t}) + \sup_{t \geq t_u^\epsilon} \;\E(e^{-u X_t}) \right] du.
\end{align}
Using that $ t \to v_t(u)$ is non-increasing for $u > z_0$, we obtain from \eqref{E:laplace} and observing that $v_{t_u^\epsilon}(u) = u^\epsilon$
\begin{align*} 
    \sup_{t\in [0, t_u^\epsilon]} \E(e^{-u X_t}) &\leq \sup_{t\in [0, t_u^\epsilon]} \exp(-x_0 v_t(u) ) \leq \exp (-x_0 u^\epsilon ),\\
    \sup_{t \geq t_u^\epsilon} \E(e^{-u X_t}) &\leq \sup_{t \geq t_u^\epsilon} \exp \left(- \int_{v_t(u)}^u \frac{F(z)}{R(z)} dz\right) 
    \leq \exp \left(- \int_{u^\epsilon}^u \frac{F(z)}{R(z)} dz\right).
\end{align*}
For $z \in [u^\epsilon, u]$
\begin{align*}
\frac{F(z)}{R(z)}  \geq \frac{2 a}{\sigma^2} \frac{1}{z(1 + \frac{\overline{\delta}^\alpha 2}{\alpha \sigma^2 z^{2 - \alpha}})+ \frac{2b}{\sigma^2}} 
\geq \frac{2 a}{\sigma^2} \frac{1}{z g(u)+ \frac{2b}{\sigma^2}}, \;  \text{with} \; 
g(u)=1 + \frac{\overline{\delta}^\alpha 2}{\alpha \sigma^2 u^{\epsilon(2 - \alpha)}}.
\end{align*}
%with
%$$
%g(u)=1 + \frac{\overline{\delta}^\alpha 2}{\alpha \sigma^2 u^{\epsilon(2 - \alpha)}}.
%$$
Hence integrating the previous inequality
\begin{align*}
    \int_{u^\epsilon}^u \frac{F(z)}{R(z)} dz 
     \geq \frac{2 a}{\sigma^2 g(u)} \ln \left(\frac{u g(u) + \frac{2b}{\sigma^2}}{u^\epsilon g(u) + \frac{2b}{\sigma^2}} \right)
     = \frac{2 a}{\sigma^2 g(u)} [\ln( u^{1- \epsilon}) + h(u)],
\end{align*}
with
$$
h(u)=\ln \left(\frac{ g(u) + \frac{2b}{u\sigma^2}}{ g(u) + \frac{2b}{u^\epsilon\sigma^2}} \right).
$$
We have $g(u)\geq 1$, $\lim_{u \to + \infty} g(u) = 1$ and $\lim_{u \to + \infty} h(u) = 0$, then
\begin{align*}
    \exp \left( - \int_{u^\epsilon}^u \frac{F(z)}{R(z)} dz \right) 
%  &  \leq \exp \left( \frac{- 2a}{\sigma^2g(u) } \ln( u^{1- \epsilon})  - \frac{2ah(u)}{\sigma^2 g(u) }\right) , \\
  &  \leq C_{\epsilon}/u^{(1-\epsilon)\frac{2a}{\sigma^2 g(u)} }.
\end{align*}
If $1 \leq p < \frac{2a}{\sigma^2}$, we choose $\epsilon > 0$ such that $p - \frac{2a}{\sigma^2} + \epsilon \frac{2a}{\sigma^2} < 0$, then we deduce
\[ \int_{u_{0, \epsilon}}^{+ \infty} u^{p-1} \exp \left(- \int_{u^\epsilon}^u \frac{F(z)}{R(z)} dz\right)du < + \infty, \]
and this allows us to conclude.

\subsubsection{Proof of Lemma \ref{L:ramenerapplati}}

We introduce a little more notation. We set
$$
a_t={\bf 1}_{\{ (i +1)\Delta_n\leq t \leq (i+2) \Delta_n \} }-{\bf 1}_{\{ i \Delta_n\leq t < (i+1) \Delta_n\} }, \quad t \in [i \Delta_n, (i+2) \Delta_n].
$$
With this notation, we have 
$$
\Delta^{s,n}_{i} X=\int_{i \Delta_n}^{(i+2) \Delta_n} a_t dX_t \quad \text{and} \quad
 \Delta^{s,n}_{i} \overline{X}=\sigma_0 X_{i \Delta_n}^{1/2} \int_{i \Delta_n}^{(i+2) \Delta_n} a_t dB_t + \delta_0^{1/ \alpha_0} X_{i \Delta_n}^{1/\alpha_0} \int_{i \Delta_n}^{(i+2) \Delta_n} a_t dL^{\alpha_0}_t.
$$
We then deduce the decomposition $ \rho_i^{n}  =  \overline{\rho}_i^{n} + \theta(1)_i^n + \theta(2)_i^n + \theta(3)_i^n$ with
\begin{align*}
    \theta(1)_i^n & = \sigma_0\frac{u_n}{\sqrt{\Delta_n}\sqrt{X_{i \Delta_n}}}  \int_{i\Delta_n}^{(i+2)\Delta_n} a_t  (X_t^{1/2} - X_{i\Delta_n}^{1/2})dB_t,\\
%    & - \int^{(i+1)\Delta_n}_{i\Delta_n}  (X_s^{1/2} - X_{i\Delta_n}^{1/2})dB_s \biggl) \\
    \theta(2)_i^n & = \delta_0^{1/\alpha_0}\frac{u_n}{\sqrt{\Delta_n}\sqrt{X_{i \Delta_n}}}  \int_{(\Delta_n}^{(i+2)\Delta_n} a_t  (X_t^{1/\alpha_0} - X_{i\Delta_n}^{1/\alpha_0})dL^{\alpha_0}_t,\\
%    & - \int^{(i+1)\Delta_n}_{i\Delta_n}  (X_s^{1/\alpha_0} - X_{i\Delta_n}^{1/\alpha_0})dL^{\alpha_0}_s \biggl) \\
    \theta(3)_i^n & = -b_0\frac{u_n}{\sqrt{\Delta_n}\sqrt{X_{i \Delta_n}}} \int_{i\Delta_n}^{(i+2)\Delta_n}  a_t(X_t - X_{i\Delta_n})dt .
%    & - \int^{(i+1)\Delta_n}_{i\Delta_n} (X_s - X_{i\Delta_n})ds \biggl) .
\end{align*}
We have the following estimates 
\begin{align}
 \forall p \in [1,\alpha_0), \; &\E_{i} (|\theta(1)_i^n|^p )  \leq  C_p u_n^p \Delta_n^{p/2}(1+ \frac{1}{X_{i\Delta_n}^{p}} ), \label{eq:theta12}\\
%    \E_{i}(|\theta(2)_i^n| ) & \leq C \frac{u_n \Delta_n^{1/\alpha - \epsilon}}{X_{i\Delta_n}^{(1 - 1/ \alpha)}}, \label{eq:abstheta2}\\ 
 \forall p \in [1,\alpha_0), \;   & \E_{i}(|\theta(2)_i^n|^p ) \leq C_p u_n^p\Delta_n^{p/\alpha_0} (1+\frac{1}{X_{i\Delta_n}^{p/2+p(1 - 1/ \alpha_0)}} + X_{i\Delta_n}^{p/\alpha_0-p/2}), \label{eq:theta22}\\
 %   \E_{i}(|\theta(3)_i^n| ) & \leq C u_n \Delta_n (1 + \frac{1}{X_{i\Delta_n}^{(1 - 1/ \alpha)}}),\label{eq:abstheta3} \\
 \forall p \in [1,\alpha_0), \;  &  \E_{i}(|\theta(3)_i^n|^p )  \leq C_p u_n^p \Delta_n^p (1+ \frac{1}{X_{i\Delta_n}^{p/2}} + X_{i\Delta_n}^{p/2}). \label{eq:theta32}
\end{align}
We deduce easily \eqref{eq:theta22} and \eqref{eq:theta32} from Proposition \ref{Th:Moments},
\eqref{eq:theta12} follows from Burkholder's inequality, Proposition \ref{Th:Moments} and the bound
$$
\forall x, y \in (0, + \infty) \; |y^{1/2}-x^{1/2}| \leq C|y-x | / x^{1/2}.
$$
 Combining the estimates \eqref{eq:theta12}, \eqref{eq:theta22}, \eqref{eq:theta32} and $|f(x+y) - f(x)| \leq C |y|$ as $f^{\prime}$ is bounded, we deduce immediately \eqref{eq:cos-cosQ}.

We now prove \eqref{eq:cos-cos}. With the previous notation, we have the decomposition
\begin{align*}
    f( \rho_i^{n}) - f(  \overline{\rho}_i^{n})  = &
     f( \rho_i^{n}) - f( \overline{\rho}_i^{n} + \theta(1)_i^n + \theta(2)_i^n) 
    +   f( \overline{\rho}_i^{n} + \theta(1)_i^n + \theta(2)_i^n) - f( \overline{\rho}_i^{n} + \theta(1)_i^n) \\
  &  +  f( \overline{\rho}_i^{n} + \theta(1)_i^n) - f( \overline{\rho}_i^{n}),
\end{align*}
and since $f^{\prime}$ is bounded it yields
$$
| \E_i (f( \rho_i^{n}) - f(  \overline{\rho}_i^{n})) | \leq C(\E_i | \theta(3)_i^n| + \E_i | \theta(2)_i^n| +| \E_i (f( \overline{\rho}_i^{n} + \theta(1)_i^n) - f( \overline{\rho}_i^{n}))|).
$$
Using  \eqref{eq:theta22} and \eqref{eq:theta32} with $p=1$
\begin{align}
   \E_i | \theta(3)_i^n| + \E_i | \theta(2)_i^n| 
     \leq C u_n \Delta_n^{1/\alpha_0} (1 + \frac{1}{X_{i\Delta_n}^{3/2 - 1/ \alpha_0}} + X_{i\Delta_n}^{1/2})   \label{pr:cos-cos1}.
\end{align}
It remains to bound $| \E_i (f( \overline{\rho}_i^{n} + \theta(1)_i^n) - f( \overline{\rho}_i^{n}))|$. Since $f^{\prime \prime}$ is bounded, we have
\begin{align*}
|\E_{i}  ( f( \overline{\rho}_i^{n} + \theta(1)_i^n) - f( \overline{\rho}_i^{n}))| \leq |\E_{i} (\theta(1)_i^n f^{\prime}( \overline{\rho}_i^{n}) )| + C \E_i |\theta(1)_i^n|^2.
\end{align*}
To  finish the proof, we show the following estimates
%%%%%
\begin{equation}\label{pr:cos-cos2}
   \E_i |\theta(1)_i^n|^2 \leq C u_n^2 \Delta_n \left(1+\frac{1}{X_{i\Delta_n}}\left[1 + \frac{1}{\Delta_n}\E_{i} \int_{i\Delta_n}^{(i+2)\Delta_n} \frac{dv}{X_v} \right]\right),
\end{equation}
and for $\epsilon >0$
\begin{align}
    | \E_{i} (\theta(1)_i^n  f^{\prime}( \overline{\rho}_i^{ n})) |
    \leq C _{\epsilon} u_n \Delta_n^{1/ \alpha_0- \epsilon} \left(1+\frac{1}{X_{i\Delta_n}^{1/2}}\left[1 + \frac{1}{\Delta_n}\E_{i} \int_{i\Delta_n}^{(i+2)\Delta_n} \frac{dv}{X_v^{1/2}} \right]\right) .\label{eq:cos-cosEtape3}
%%%%%% a corriger
\end{align}
Collecting \eqref{pr:cos-cos1}, \eqref{pr:cos-cos2} and \eqref{eq:cos-cosEtape3}, we deduce \eqref{eq:cos-cos}.

\noindent
\underline{Proof of \eqref{pr:cos-cos2}}
We write It\^{o}'s formula for the function $x \to \sqrt{x}$ which is $\mathcal{C}^2$ on $(0, +\infty)$ 
\[ \sqrt{X_t} - \sqrt{X_s} = \int_s^t \overline{b}_v dv + \int_s^t \frac{\sigma_0}{2} dB_v
    + \int_s^t \int \overline{c}_{v-} \tilde{N}(dv, dz),\]
where
\begin{align*}
    \overline{b}_v & = \frac{(a_0-b_0X_v )}{2 X_v^{1/2}} - \frac{\sigma_0^2 }{8 X_v^{1/2}} 
     + \int( \sqrt{X_{v-} + \delta_0^{1/\alpha_0} X_{v-}^{1/\alpha_0}z} - \sqrt{X_{v-}} - \frac{1}{2 X_{v-}^{1/2}} \delta_0^{1/\alpha_0} X_{v-}^{1/\alpha_0}z ) F^{\alpha_0}(z) dz, \\
    \overline{c}_{v-} & = (X_{v-} + \delta_0^{1/\alpha_0} X_{v-}^{1/\alpha_0}z)^{1/2} - (X_{v-})^{1/2}.
\end{align*}
%We observe that $|(X_v + \delta_0^{1/\alpha_0} X_v^{1/\alpha_0}z)^{1/2} - (X_v)^{1/2} - \frac{1}{2 X_v^{1/2}} \delta_0^{1/\alpha_0} X_v^{1/\alpha_0}z| \leq C z^2 \frac{X_v^{2/\alpha_0}}{X_v^{3/2}}$, and conclude that 
Moreover, we check that $|\overline{b}_v| \leq C (1 +\frac{1}{\sqrt{X_v}} + \sqrt{X_v})$ and that 
$
\vert \overline{c}_v | \leq C(1+\sqrt{X_v}) (\sqrt{z}{\bf 1}_{\{z>1\}} +z {\bf 1}_{\{0<z \leq 1\}}).
$
%$( \int_0^t \int \overline{c}_{v-} d\tilde{N}(dv, dz))_t$ is a square integrable martingale.
We obtain  the decomposition
\begin{equation} \label{eq:Dtheta1}
\theta(1)_i^n=\mu(1)_i^n + \mu(2)_i^n + \mu(3)_i^n ,
\end{equation}
with
\begin{align*}
 \mu(1)_i^n   &  = \frac{u_n \sigma_0}{\sqrt{\Delta_n X_{i \Delta_n} }} \int_{i\Delta_n}^{(i+2)\Delta_n} a_t (\int_{i\Delta_n}^t \overline{b}_v dv )dB_t, \\
\mu(2)_i^n &    = \frac{u_n \sigma_0}{\sqrt{\Delta_n X_{i \Delta_n}}} \frac{\sigma_0}{2} \int_{i\Delta_n}^{(i+2)\Delta_n} a_t (B_t - B_{i\Delta_n})dB_t, \\
\mu(3)_i^n   &  = \frac{u_n \sigma_0}{\sqrt{\Delta_n X_{i \Delta_n}}} \int_{i\Delta_n}^{(i+2)\Delta_n} a_t\left( \int_{i \Delta_n}^t \int \overline{c}_{v-} \tilde{N}(dv, dz) \right) dB_t.
\end{align*}
From Burkholder's inequality  and Proposition  \ref{Th:Moments} (ii), we have
\begin{align}
\forall p \in & [1,2], \,  \E_i | \mu(1)_i^n|^p \leq C_p u_n^p \Delta_n^p \left(1+\frac{1}{X_{i\Delta_n}^{p/2}}\left[1 + \frac{1}{\Delta_n}\E_{i} \int_{i\Delta_n}^{(i+2)\Delta_n} \frac{dv}{X_v^{p/2}} \right]\right),  \label{eq:mu1}\\
\forall p & \geq 1 ,\,  \E_i | \mu(2)_i^n|^p \leq C_p \frac{u_n^p }{X_{i\Delta_n}^{p/2}} \Delta_n^{p/2} , \label{eq:mu2}\\
\forall p & \in (\alpha_0,2], \;  \E_i | \mu(3)_i^n|^p \leq C_p u_n^p \Delta_n \left(1+\frac{1}{X_{i\Delta_n}^{p/2}} \right). \label{eq:mu3}
\end{align}
Note that to obtain \eqref{eq:mu3}, we have used 
$$
\E_i \sup_{i \Delta_n \leq t \leq (i+2) \Delta_n} |\int_{i \Delta_n}^t \int \overline{c}_{v-}  \tilde{N}(dv, dz)|^p\leq C_p \E_i \int_{i\Delta_n}^{(i+2)\Delta_n} \int |\overline{c}_{v-}|^p \overline{N}(dv,dz),
$$
combined with $\vert \overline{c}_v | \leq C(1+\sqrt{X_v}) (\sqrt{z}{\bf 1}_{\{z>1\}} +z {\bf 1}_{\{0<z \leq 1\}})$. We then deduce \eqref{pr:cos-cos2} from the previous bounds with $p=2$.
%%%%reprendre ici

\noindent
\underline{Proof of \eqref{eq:cos-cosEtape3}}
Since $f^{\prime}$ is bounded, we have using \eqref{eq:Dtheta1} and H\"older's inequality with $\alpha_0 <p\leq 2$
\begin{equation} \label{eq:mu2Wc}
  | \E_{i} (\theta(1)_i^n  f^{\prime}( \overline{\rho}_i^{ n})) | \leq C(\E_i | \mu(1)_i^n| + |\E_{i} (\mu(2)_i^n  f^{\prime}( \overline{\rho}_i^{ n})) | + \E^{1/p}_i | \mu(3)_i^n|^p).
\end{equation}
It remains to consider $\E_{i} (\mu(2)_i^n  f^{\prime}( \overline{\rho}_i^{ n}))$. We split $\overline{\rho}_i^{ n}$ in two parts
(Brownian part and pure-jump part)
$$
\overline{\rho}_i^{n}=\frac{\sigma_0  u_n}{\sqrt{\Delta_n}}\Delta^{s,n}_{i} B  + 
\frac{\delta_0^{1/\alpha_0} X_{i\Delta_n}^{1/\alpha_0} u_n}{\sqrt{\Delta_n} \sqrt{X_{i \Delta_n}}} \Delta^{s,n}_{i} L^{\alpha_0} := \overline{\rho}_i^{1, n}+\overline{\rho}_i^{2, n},
$$
and write 
$$
\E_{i} (\mu(2)_i^n f^{\prime}( \overline{\rho}_i^{ n}))=\E_i (\mu(2)_i^nf^{\prime}( \overline{\rho}_i^{1, n}))+ \E_i  (\mu(2)_i^n (f^{\prime}( \overline{\rho}_i^{ n})-f^{\prime}( \overline{\rho}_i^{1, n}))).
$$
Since $f^{\prime}$ is odd, an explicit computation gives $\E_i (\mu(2)_i^nf^{\prime}( \overline{\rho}_i^{1, n}))=0$. For the second term, we use successively  that $f^{\prime \prime}$ is bounded, H\"older's inequality with $1 < q <\alpha_0$ and the scaling property of $L^{\alpha_0}$ to obtain
\begin{align}
\E_i | \mu(2)_i^n (f^{\prime}( \overline{\rho}_i^{ n})-f^{\prime}( \overline{\rho}_i^{1, n}))| & \leq C \E_i | \mu(2)_i^n\overline{\rho}_i^{2, n} | 
 \leq C_q \frac{u_n}{\sqrt{\Delta_n}}(1+ X_{i \Delta_n}^{1/2})  \E_i^{\frac{1}{q'}} | \mu(2)_i^n|^{q'} \E^{\frac{1}{q}}_i | \Delta^{s,n}_{i} L^{\alpha_0} |^q  \nonumber
\\
&  \leq C_q u_n^2(1+ \frac{1}{X_{i \Delta_n}^{1/2}} )\Delta_n^{1/ \alpha_0}. \label{eq:mu2W}
\end{align}
We finally deduce  \eqref{eq:cos-cosEtape3} from  \eqref{eq:mu2Wc} (using \eqref{eq:mu1} with $p=1$ and \eqref{eq:mu3} with  $p=\alpha_0+ \epsilon$) and \eqref{eq:mu2W}.

\subsubsection{Proof of Lemma \ref{L:expan}}

We assume that $\theta=(\sigma^2, \delta, \alpha) \in A$ a compact subset of $\Theta$. 
We denote by $g_{\alpha}$ the density of $S_1^{\alpha}$, by $\Phi$ the density of the standard Gaussian distribution and we set 
$a_n(x, \theta)=(2 \delta )^{1/\alpha} u_n (x \Delta_n)^{1/\alpha- 1/2}$. Then
\begin{eqnarray*}
P_n f(x, \theta )  
%= \E f\left (\sqrt{2} u_n \sigma B_1 + a_n(x, \theta) S^{\alpha}_1\right) 
%\quad \quad \\
  = 2 \int_{\R_+} f(v) dv \int_{\R} g_{\alpha} \left( \frac{v-y}{a_n(x, \theta) }\right)
 \frac{1}{a_n(x, \theta)}
  \Phi\left( \frac{y}{\sqrt{2} u_n \sigma }\right) \frac{1}{\sqrt{2} u_n \sigma }dy.
\end{eqnarray*}
Using successively a change of variables and Fubini's Theorem yields
\begin{align*}
P_n f(x, \theta ) &  = 2 \int_{\R} \Phi(z )dz \int_{\R_+} f(v) g_{\alpha} \left( \frac{v-\sqrt{2} u_n \sigma z}{a_n(x, \theta) }\right)
 \frac{1}{a_n(x, \theta)} dv .
\end{align*}
Recalling that $f=0$ on $[-\eta, \eta]$, we split the previous integral in two parts $P_n f(x, \theta )=I_n^{1} + I_n^{2}$ where
\begin{align*}
I_n^{1} =2 \int_{\{ |z] > \frac{\eta}{2 \sqrt{2} u_n \sigma}\}} \Phi(z )dz \int_{\R_+} f(v) g_{\alpha} \left( \frac{v-\sqrt{2} u_n \sigma z}{a_n(x, \theta) }\right)
 \frac{1}{a_n(x, \theta)} dv, \\
I_n^{2}= 2 \int_{\{ |z] \leq \frac{\eta}{2 \sqrt{2} u_n \sigma}\}} \Phi(z )dz \int_{\{v>\eta\}} f(v) g_{\alpha} \left( \frac{v-\sqrt{2} u_n \sigma z}{a_n(x, \theta) }\right)
 \frac{1}{a_n(x, \theta)} dv.
\end{align*}
By assumption, $u_n^2=\overline{u}_n /\ln(1/ \Delta_n)$ with $\overline{u}_n \rightarrow 0$, so we deduce immediately, since $f$ is bounded, that
\begin{equation} \label{eq:BI1}
| I_n^1| \leq C \exp(-\frac{\eta^2}{ 16 u_n^2 \sigma^2 }) \leq C \Delta_n^{C' /\overline{u}_n},
\end{equation}
where $C$ and $C'$ do not depend on $\theta$. Turning to $I_n^2$, we observe that for $|z| \leq \frac{\eta}{2 \sqrt{2} u_n \sigma}$ and $v >\eta$,
$$
\frac{v-\sqrt{2} u_n \sigma z}{a_n(x, \theta) } \geq \frac{\eta}{2 a_n(x, \theta)} >0.
$$ 
Using the asymptotic expansion of $g_{\alpha}$ (see Sato \cite{Sato}), we have
$$
\forall y >0 \quad | g_{\alpha} (y)- \frac{c_{\alpha}}{2y^{\alpha+1}}| \leq \frac{C}{y^{2 \alpha+1}},  \; \text{with} \; c_{\alpha} \; \text{ given by \eqref{eq:c-alpha}}
$$
and we deduce
$$
g_{\alpha} \left( \frac{v-\sqrt{2} u_n \sigma z}{a_n(x, \theta) }\right)=\frac{c_{\alpha} a_n(x, \theta)^{\alpha+1}}{2 v^{\alpha+1}} \frac{1}{(1-\sqrt{2} \sigma u_n z/v)^{\alpha+1}}(1+ R^0_n(z,v)),
$$
with 
$$
|R^0_n(z,v)| \leq C\frac{a_n(x, \theta)^{\alpha}}{ v^{\alpha}} \frac{1}{(1-\sqrt{2} \sigma u_n z/v)^{\alpha}} \leq C\frac{a_n(x, \theta)^{\alpha}}{ v^{\alpha}}, 
$$
where we used $| \sqrt{2} \sigma u_n z/v| \leq 1/2$ for the last inequality. Moreover, we have the first order expansion
$$
\frac{1}{(1-\sqrt{2} \sigma u_n z/v)^{\alpha+1}}=1+ R_n^1(z,v), \; \text{with} \; |R_n^1(z,v)| \leq C u_n z/v.
$$
Reporting these results in $I_n^2$, we obtain
\begin{eqnarray*}
I_n^2  = &  \int_{\{ |z] \leq \frac{\eta}{2 \sqrt{2} u_n \sigma}\}} \Phi(z )dz \int_{\{v>\eta\}} f(v) \frac{c_{\alpha} a_n(x, \theta)^{\alpha}}{ v^{\alpha+1}} (1+R_n^1(z,v))(1+ R^0_n(z,v))dv \\
 = & c_{\alpha} a_n(x, \theta)^{\alpha}\left( \int_{\{v>\eta\}}  \frac{f(v)}{ v^{\alpha+1}} dv +  \int_{\{ |z] > \frac{\eta}{2 \sqrt{2} u_n \sigma}\}} \Phi(z )dz \int_{\{v>\eta\}}  \frac{f(v)}{ v^{\alpha+1}} dv \right. \\
& \left.+ \int_{\{ |z] \leq \frac{\eta}{2 \sqrt{2} u_n \sigma}\}} \Phi(z )dz \int_{\{v>\eta\}}\frac{f(v) }{ v^{\alpha+1}} [R_n^1(z,v)(1+ R^0_n(z,v))+R^0_n(z,v)]dv \right).
\end{eqnarray*}
This finally gives
\begin{equation} \label{eq:In2}
I_n^2 = c_{\alpha} a_n(x, \theta)^{\alpha} \left( \int_{\{v>\eta\}}  \frac{f(v)}{ v^{\alpha+1}} dv +R_n \right),
\end{equation}
with 
$$
| R_n| \leq C (\Delta_n^{C' /\overline{u}_n} + u_n + u_n^{\alpha} (x \Delta_n)^{1- \alpha/2}) \leq C(1 + \sqrt{x}) (\Delta_n^{C' /\overline{u}_n}+u_n).
$$
From \eqref{eq:BI1} and \eqref{eq:In2}, we deduce the result of Lemma \ref{L:expan}.

\subsubsection{Proof of Proposition \ref{Th:TCL}}

The proof consists first to replace, in the expression of $F_n(\theta_0)$,  $\rho_{2i}^{n}$ by $\overline{\rho}_{2i}^{n}$ (defined by \eqref{eq:rho}) by using Lemma \ref{L:ramenerapplati} and next to study the convergence in distribution of 
$$
A_n(\theta_0) \frac{1}{n} \sum_{i=0}^{\lfloor \frac{n}{2} \rfloor - 1} [ f( \overline{\rho}_{2i}^n) 
- P_n f(X_{2i \Delta_n}, \theta_0) ].
$$
We set 
\begin{equation} \label{E:zeta}
 \zeta_{2i}^{n} = A_n(\theta_0) \frac{1}{n} [ f (\overline{\rho}_{2i}^n) 
- P_n f(X_{2i \Delta_n}, \theta_0)],
\end{equation}
hence $\zeta_{2i}^{n} =(\zeta_{2i}^{k,n})_{1 \leq k \leq 3} $ with
\begin{align*}
    \zeta_{2i}^{1,n} & = \frac{1}{u_n^2 \sqrt{n}} [ f_1 (\overline{\rho}_{2i}^n) 
- P_n f_1(X_{2i \Delta_n}, \theta_0) ], \\
%    \zeta_{2i}^{2,n} & = \frac{1}{u_n^{\alpha_0/2}\Delta_n^{1/2 - \alpha_0/4} \sqrt{n}} [ f_2(\overline{\rho}_{2i}^n) 
%- P_n f_2 (X_{2i \Delta_n}, \theta_0) ],\\
    \zeta_{2i}^{k,n} & = \frac{1}{u_n^{\alpha_0/2}\Delta_n^{1/2 - \alpha_0/4} \sqrt{n}} [ f_k( \overline{\rho}_{2i}^n) 
- P_n f_k (X_{2i \Delta_n}, \theta_0)], \quad \text{for} \quad k \in \{2,3\}.
\end{align*}

\noindent
1. \underline{ Case $n \Delta_n =1$.}
We first remark that for any $0 < p \leq 1$ we have the convergence in probability
$$
\frac{1}{n} \sum_{i=0}^{\lfloor \frac{n}{2} \rfloor - 1} X_{2 i \Delta_n}^p \xrightarrow[n \to \infty]{\PP} \frac{1}{2}\int_0^1 X_t^p dt.
$$
 Moreover from Proposition \ref{Th:Moments} and recalling  $a_0 > \sigma_0^2$, we have  for $1 \leq p \leq 2$
\begin{align} \label{E:momneg1}
\E\left( \frac{1}{n}  \sum_{i=0}^{\lfloor \frac{n}{2} \rfloor - 1} \frac{1}{X_{i\Delta_n}^p} \right) \leq  \sup_{t \geq 0} \E (\frac{1}{X_t^p}) < + \infty, 
\end{align} 
and using $1/(X_{i\Delta_n}X_v) \leq 1/(2X_{i\Delta_n}^2) +1/(2X_v^2) $
\begin{align} \label{E:momneg2}
\E\left( \frac{1}{n} \sum_{i=0}^{\lfloor \frac{n}{2} \rfloor - 1} \frac{1}{\Delta_n}\E_{i} \int_{i\Delta_n}^{(i+2)\Delta_n} \frac{dv}{X_{i\Delta_n}X_v} \right) \leq 
\sup_{t \geq 0} \E (\frac{1}{X_t^2}) < + \infty.
\end{align}
Next from Lemma \ref{L:ramenerapplati}  and  choosing $\epsilon>0$ such that 
$
1/\alpha_0- \epsilon >1/2  $
and $1/\alpha_0- \epsilon  >1- \alpha_0/4$,
which is possible since $1 < \alpha_0<2$, we have the convergence in probability
\begin{equation*} 
  A_n(\theta_0) \frac{1}{n} \sum_{i=0}^{\lfloor \frac{n}{2} \rfloor - 1} | \E_{2i} (f( \rho_{2i}^{n}) - f( \overline{\rho}_{2i}^{n}) ) | \to 0.
\end{equation*}
Observing that for $f$ bounded 
$\E_{2i} | f( \overline{\rho}_{2i}^n) - P_n f(X_{2i \Delta_n}, \theta_0) |^{2} \leq  C \E_{2i} | f( \overline{\rho}_{2i}^n) - P_n f(X_{2i \Delta_n}, \theta_0) |^p$ with $p< \alpha_0$, we check with  Lemma \ref{L:ramenerapplati}
\begin{align*}
\frac{n}{u_n^4} \frac{1}{n^2} & \sum_{i=0}^{\lfloor \frac{n}{2} \rfloor - 1} \E_{2i} | f_1( \rho_{2i}^{n}) - f_1( \overline{\rho}_{2i}^{n}) |^2\to 0, \\
 \frac{n}{u_n^{\alpha_0}\Delta_n^{1 - \alpha_0/2}} & \frac{1}{n^2}  \sum_{i=0}^{\lfloor \frac{n}{2} \rfloor - 1} \E_{2i} | f_k( \rho_{2i}^{n}) - f_k( \overline{\rho}_{2i}^{n}) |^2\to 0 \quad k \in \{2,3\}.
\end{align*}
Consequently, we have the convergence in probability 
\begin{align} \label{E:app}
    A_n(\theta_0) \frac{1}{n} \sum_{i=0}^{\lfloor \frac{n}{2} \rfloor - 1} [ f( \rho_{2i}^{n}) - f( \overline{\rho}_{2i}^{n}) ] \to 0,
\end{align}
and the proof of  Proposition \ref{Th:TCL} reduces to show the stable convergence
\[ A_n(\theta_0) \frac{1}{n} \sum_{i=0}^{\lfloor \frac{n}{2} \rfloor - 1} [ f( \overline{\rho}_{2i}^n) 
- P_n f(X_{2i \Delta_n}, \theta_0) ]\xRightarrow[n \to \infty]{\mathcal{L}-s} \Sigma(\theta_0)^{1/2}\mathcal{N},\]
for $\mathcal{N}$ a standard Gaussian variable independent of $\Sigma(\theta_0)$ defined by \eqref{E:sigma}. 
According to Theorem IX.7.26 in Jacod and Shiryaev \cite{JacodShiryaev}, it is sufficient to show the following convergences in probability
(with $\zeta_{2i}^{ n}$ defined by \eqref{E:zeta})
\begin{align}
    & \sum_{i=0}^{\lfloor \frac{n}{2} \rfloor - 1}  \E_{2i} (\zeta_{2i}^{ n} ) \to 0,  \label{eq:point1cvTCL} \\
    & \sum_{i=0}^{\lfloor \frac{n}{2} \rfloor - 1}   \E_{2i} (|\zeta_{2i}^{k, n}|^{p}) \to 0   \quad \text{for some} \; p>2,  \quad k=1,2,3, \label{eq:point2cvTCL} \\
    & \sum_{i=0}^{\lfloor \frac{n}{2} \rfloor - 1}   \E_{2i} (\zeta_{2i}^{k, n} \zeta_{2i}^{l, n}) \to \Sigma_{k,l}(\theta_0) 
     \quad \text{for} \quad k,l=1,2,3, \label{eq:point3cvTCL} \\
    & \sum_{i=0}^{\lfloor \frac{n}{2} \rfloor - 1}  \E_{2i}(\zeta_{2i}^{k, n}  (M_{2(i+1)\Delta_n} -M_{2i \Delta_n})) \to 0  \quad \text{for} \quad k=1,2,3,  \label{eq:point4cvTCL}
\end{align}
where the last convergence holds for any square integrable martingale $M$.
The convergence \eqref{eq:point1cvTCL} is immediate since $\E_{2i}  f( \overline{\rho}_{2i}^n) 
=P_n f(X_{2i \Delta_n}, \theta_0) $ and  \eqref{eq:point4cvTCL} is addressed in Lemma \ref{L:AccM} below.

We deduce  \eqref{eq:point2cvTCL} from Lemma \ref{L:expan} for $k=2$ and $k=3$. For $k=1$,
we recall that using the notation \eqref{eq:rho} with $f_1(x)= \cos(x)$
\begin{equation} \label{eq:cosrho}
P_nf_1(X_{2i \Delta_n}, \theta_0)= \E_{2i} f_1( \overline{\rho}_{2i}^{n}) =\exp(-u_n^2 \sigma_0^2- 2 \delta_0 X_{2 i \Delta_n}^{1-1/ \alpha_0}u_n^{\alpha_0} \Delta_n^{1-1/ \alpha_0}),
\end{equation}
and
$$
\E_{2i} f_1( 2\overline{\rho}_{2i}^{n})=\exp(-4u_n^2 \sigma_0^2- 2^{1+ \alpha_0} \delta_0 X_{2 i \Delta_n}^{1-1/ \alpha_0}u_n^{\alpha_0} \Delta_n^{1-1/ \alpha_0}).
$$
Using $\cos^2(a) = (1+ \cos(2a) )/2$, we deduce
$$
\E_{2i} (f_1( \overline{\rho}_{2i}^{n})^2)  = (1+ \E_{2i} f_1( 2\overline{\rho}_{2i}^{n}) )/2,
$$
and for $q \geq 2$, since $f_1$ is bounded
\begin{equation}
 \E_{2i }(|f_1(\overline{\rho}_{2i}^{n}) - P_nf_1(X_{2i \Delta_n}, \theta_0)  |^q )  \leq  \E_{2i} (f_1(\rho_{2i}^n)  - \E_{2i} f_1( \overline{\rho}_{2i}^{n}))^2    =  \frac{\E_{2i} f_1( 2\overline{\rho}_{2i}^{n}) - 2 \E^2_{2i} f_1( \overline{\rho}_{2i}^{n}) +1}{2}. \label{eq:cos2}
\end{equation}
We now expand $\E_{2i} f_1( \overline{\rho}_{2i}^{n})$ and $\E_{2i} f_1( 2\overline{\rho}_{2i}^{n})$ around 0 and get that
\begin{equation} \label{eq:expcos}
\E_{2i} f_1( 2\overline{\rho}_{2i}^{n}) - 2 \E^2_{2i} f_1( \overline{\rho}_{2i}^{n}) +1 = 4 u_n^4 \sigma_0^4 + R_n(x, \theta_0), 
\; \text{with} \; |R_n(x, \theta_0)| \leq C(u_n^6+ u_n^{\alpha_0} \Delta_n^{1 - \alpha_0/2})(1 + x).
\end{equation}
%with 
%$$
%|R_n(x, \theta_0)| \leq C(u_n^6+ u_n^{\alpha_0} \Delta_n^{1 - \alpha_0/2})(1 + x).
%$$
This gives \eqref{eq:point2cvTCL} for $k=1$.

%%%%%%%
We now prove \eqref{eq:point3cvTCL}. From \eqref{eq:cos2} and \eqref{eq:expcos} we have
\begin{align*}
 \sum_{i=0}^{\lfloor \frac{n}{2} \rfloor - 1}   \E_{2i} (\zeta_{2i}^{1, n} )^2  = \frac{1}{n}
    \frac{1}{u_n^4}  \sum_{i=0}^{\lfloor \frac{n}{2} \rfloor - 1}   \frac{\E_{2i} f_1( 2\overline{\rho}_{2i}^{n}) - 2 \E^2_{2i} f_1( \overline{\rho}_{2i}^{n}) +1}{2}
      \to  \sigma_0^4.
\end{align*}
We then apply Lemma \ref{L:expan} to $f_k $, $f_l$  and $f_k f_l$ with $k, l \in \{2,3\}$.  It yields
\begin{align*}
\E_{2i} &[ f_k ( \overline{\rho}_{2i}^n) 
- P_n f_k (X_{2i \Delta_n}, \theta_0) ] [ f_l ( \overline{\rho}_{2i}^n) 
- P_n f_l (X_{2i \Delta_n}, \theta_0) ] \\
&=P_n (f_k f_l )(X_{2i \Delta_n}, \theta_0)
- P_n f_k (X_{2i \Delta_n}, \theta_0) P_n f_l (X_{2i \Delta_n}, \theta_0)  \\
&=
u_n^{\alpha_0} \Delta_n^{1-\alpha_0/2}c_{\alpha_0} \delta_0 X_{2i \Delta_n}^{1-\alpha_0/2} \int \frac{f_k f_l (v)}{|v|^{\alpha_0+1}} dv 
 +u_n^{\alpha_0} \Delta_n^{1-\alpha_0/2} \varepsilon_n(\theta_0,X_{2 i \Delta_n}),  
\end{align*}
where
$$
 |\varepsilon_n(\theta_0, X_{2 i \Delta_n})| \leq (1+X_{2 i \Delta_n}) \varepsilon_n, \quad \text{with} \quad \varepsilon_n \rightarrow 0.
$$
Consequently, we deduce for $k,l \in \{2,3\}$
\begin{align*}
 \sum_{i=0}^{\lfloor \frac{n}{2} \rfloor - 1}   \E_{2i} (\zeta_{2i}^{k, n} \zeta_{2i}^{l, n}) \to 
 \frac{1}{2}c_{\alpha_0} \delta_0 \int_0^1X_{s}^{1-\alpha_0/2} ds \int \frac{f_k f_l (v)}{|v|^{\alpha_0+1}} dv .
 \end{align*}
 We finally   address the cross terms $ \sum_{i=0}^{\lfloor \frac{n}{2} \rfloor - 1}   \E_{2i} (\zeta_{2i}^{1, n} \zeta_{2i}^{k, n})$ for $k \in \{2,3\}$.
 We have
 \begin{align*}
 \E_{2i} & [ (f_1( \overline{\rho}_{2i}^n) 
- P_n f_1(X_{2i \Delta_n}, \theta_0)] [f_k( \overline{\rho}_{2i}^n) 
- P_n f_k (X_{2i \Delta_n}, \theta_0) )] \\
& =P_n (f_1 f_k)(X_{2i \Delta_n}, \theta_0)
- P_n f_1 (X_{2i \Delta_n}, \theta_0) P_n f_k (X_{2i \Delta_n}, \theta_0) . 
\end{align*}
Using  again Lemma \ref{L:expan} for $f_1 f_k$ and $f_k$ with $k \in \{2,3\}$ and observing that $P_n f_1 (X_{2i \Delta_n}, \theta_0)$ is bounded
we have
$$
|P_n (f_1 f_k)(X_{2i \Delta_n}, \theta_0)
- P_n f_1 (X_{2i \Delta_n}, \theta_0) P_n f_k (X_{2i \Delta_n}, \theta_0)| \leq C u_n^{\alpha_0} \Delta_n^{1-\alpha_0/2} (1+X_{2 i \Delta_n}),
$$
and we obtain for $k \in \{2,3\}$,
$$
\sum_{i=0}^{\lfloor \frac{n}{2} \rfloor - 1}   \E_{2i} (\zeta_{2i}^{1, n} \zeta_{2i}^{k, n}) \to 0.
$$
We finally conclude that
$$
\sum_{i=0}^{\lfloor \frac{n}{2} \rfloor - 1}   \E_{2i} (\zeta_{2i}^{ n} (\zeta_{2i}^{n})^T) \to \Sigma(\theta_0).
$$
We have proved equation \eqref{eq:point3cvTCL}. 

We check easily  that $\Sigma(\theta_0)$ is definite positive and we finish the proof of Proposition \ref{Th:TCL} in the case $n \Delta_n=1$ with Lemma \ref{L:AccM} below. 
%%%%

%%%%
\begin{lem}\label{L:AccM} Let $M$ be a square-integrable martingale. Let $f_1$ and $f_2$ be even bounded functions. We assume moreover that $f_1$ is $\mathcal{C}^1$ with bounded derivative,  $f_2=0$ on $[-\eta, \eta]$, for $\eta>0$ and that $\Delta_n = 1/n$ and $u_n=1/[\ln(1/ \Delta_n)]^{p}$ with $p>1/2$. Then we have the convergence in probability (where $\overline{\rho}_{2i}^n$ is defined by \eqref{eq:rho} and $\Delta_{2i}^{2,n}M= M_{2(i+1)\Delta_n} -M_{2i \Delta_n}$)
\[ (i) \; \frac{1}{u_n^2 \sqrt{n}} \sum_{i=0}^{\lfloor \frac{n}{2} \rfloor - 1} \E_{2i} [( f_1( \overline{\rho}_{2i}^n) 
- P_n f_1(X_{2i \Delta_n}, \theta_0) ) \Delta_{2i}^{2,n}M ]  \to 0,\]
\[ (ii) \; \frac{1 }{u_n^{\alpha_0/2}\Delta_n^{1/2 - \alpha_0/4} \sqrt{n}} \sum_{i=0}^{\lfloor \frac{n}{2} \rfloor - 1} \E_{2i}  [(f_2 ( \overline{\rho}_{2i}^n) 
- P_n f_2 (X_{2i \Delta_{n}}, \theta_0)) \Delta_{2i}^{2,n}M ]  \to 0.\]
\end{lem}

\begin{proof}
From Theorem 2.2.15 in \cite{JacodProtter}, we only consider the cases where $M$ is the Brownian Motion $B$ or $M$ is orthogonal to $B$.

(i) Proceeding as in the proof of  \eqref{eq:cos-cosEtape3}, we separate the Brownian component and the pure-jump component  $\overline{\rho}_{2i}^n=\overline{\rho}_{2i}^{1,n}+\overline{\rho}_{2i}^{2,n}$ and  since $f_1$ and $f_1^{\prime}$ are bounded, for $1 \leq p<\alpha_0$
$$
\E_{2i} | f_1( \overline{\rho}_{2i}^n) -f_1( \overline{\rho}_{2i}^{1,n}) |^2 \leq C_p u_n^p \Delta_n^{p(1/ \alpha_0-1/2)} X_{2i \Delta_n}^{p(1/ \alpha_0-1/2)}.
$$
Choosing $p$ close to $\alpha_0$ and using Proposition \ref{Th:Moments} (ii), it yields for $\epsilon >0$
$$
\E | f_1( \overline{\rho}_{2i}^n) -f_1( \overline{\rho}_{2i}^{1,n}) |^2 \leq C_{\epsilon} u_n^{\alpha_0-\epsilon} \Delta_n^{1- \alpha_0/2 - \epsilon}. 
$$
Setting
$$
S_n^1=\frac{1}{u_n^2 \sqrt{n}} \sum_{i=0}^{\lfloor \frac{n}{2} \rfloor - 1} \E_{2i} [( f_1( \overline{\rho}_{2i}^n) 
- f_1( \overline{\rho}_{2i}^{1,n}) ) \Delta_{2i}^{2,n}M ],
$$
we deduce  from Cauchy-Schwarz inequality and the previous bound with $\epsilon$ such that  $1- \alpha_0/2-\epsilon>0$
\begin{align*}
\E|S_n^{1}| &\leq \frac{C_{\epsilon} }{u_n^2 \sqrt{n}} \left(\sum_{i=0}^{\lfloor \frac{n}{2} \rfloor - 1} u_n^{\alpha_0-\epsilon} \Delta_n^{1- \alpha_0/2 - \epsilon}
 \sum_{i=0}^{\lfloor \frac{n}{2} \rfloor - 1} \E (\Delta_{2i}^{2,n}M)^2 \right)^{1/ 2}
  \to 0,
\end{align*}
observing that the second sum is bounded since $M$ is a square-integrable martingale. So the proof of (i) reduces to $S_n^2 \to 0$ where
$$
S_n^2=\frac{1}{u_n^2 \sqrt{n}} \sum_{i=0}^{\lfloor \frac{n}{2} \rfloor - 1} \E_{2i} [ f_1( \overline{\rho}_{2i}^{1,n}) 
 \Delta_{2i}^{2,n}M ].
$$
If $M=B$, using that $f_1$ is even and  
recalling that $\overline{\rho}_{2i}^{1,n}=\frac{\sigma_0  u_n}{\sqrt{\Delta_n}}\Delta^{s,n}_{2i} B $, an explicit calculus gives $\E_{2i} [ f_1( \overline{\rho}_{2i}^{1,n}) 
 \Delta_{2i}^{2,n}M ]=0$. If $M$ is a square integrable martingale orthogonal to $B$, we also have $\E_{2i} [ f_1( \overline{\rho}_{2i}^{1,n}) 
 \Delta_{2i}^{2,n}M ]=0$ from the representation Theorem and It\^{o}'s Formula. 
 
 (ii) We follow the proof of Lemma 8 in \cite{Bull}. We first assume that $M=B$. From H\"{o}lder's inequality with $p>1$, we have
 $$
 |\E_{2i}[f_2 ( \overline{\rho}_{2i}^n) 
 \Delta_{2i}^{2,n}B ]| \leq C\E^{1/p}_{2i}(f_2 ( \overline{\rho}_{2i}^n))^p \sqrt{\Delta_n} \leq C\Delta_n^{(1- \alpha_0/2)/p} \sqrt{\Delta_n}(1+ X_{2i \Delta_n}), 
 $$
 where for the second inequality we applied  Lemma \ref{L:expan} to $f_2^p$. We deduce then (ii) in that case choosing $p<2$.
 
 We now assume that $M$ is orthogonal to $B$. From the representation Theorem (we refer to Chapter III.4 of \cite{JacodShiryaev}), we have for some predictable process $(G_t)$
 $$
 M_t =\int_0^t \int G_s \tilde{N}(ds,dz) \; \text{with} \; \E \int_0^1 \int G_s^2 \overline{N}(ds,dz) < \infty,
 $$
 and also with $(h_t)$ and $(H_t)$ predictable and $(H_t)$ bounded (since $f_2$ is bounded)
 $$
 f_2 ( \overline{\rho}_{2i}^n) 
- P_n f_2 (X_{2i \Delta_{n}}, \theta_0)=\int_{2i \Delta_n}^{2(i+1) \Delta_n} h_s dB_s + \int_{2i \Delta_n}^{2(i+1) \Delta_n} \int H_s \tilde{N}(ds,dz).
 $$ 
 We then deduce from It\^{o}'s Formula
 $$
  \E_{2i}  [(f_2 ( \overline{\rho}_{2i}^n) 
- P_n f_2 (X_{2i \Delta_{n}}, \theta_0)) \Delta_{2i}^{2,n}M ]= \E_{2i}  \int_{2i \Delta_n}^{2(i+1) \Delta_n}  \int H_s G_s  \overline{N}(ds,dz).
 $$
 With the decomposition $G_s =G_s{\bf 1}_{| G_s| > \epsilon_n}  +G_s {\bf 1}_{| G_s| \leq \epsilon_n} $, the proof reduces to study  the convergence of
 \begin{align*}
 T^1_n= \frac{1 }{u_n^{\alpha_0/2}\Delta_n^{1/2 - \alpha_0/4} \sqrt{n}} \sum_{i=0}^{\lfloor \frac{n}{2} \rfloor - 1} \E_{2i}  \int_{2i \Delta_n}^{2(i+1) \Delta_n}  \int H_s G_s {\bf 1}_{| G_s| > \epsilon_n}  \overline{N}(ds,dz), \\
  T^2_n= \frac{1 }{u_n^{\alpha_0/2}\Delta_n^{1/2 - \alpha_0/4} \sqrt{n}} \sum_{i=0}^{\lfloor \frac{n}{2} \rfloor - 1} \E_{2i}  \int_{2i \Delta_n}^{2(i+1) \Delta_n}  \int H_s G_s {\bf 1}_{| G_s| \leq \epsilon_n}  \overline{N}(ds,dz).
 \end{align*}
Since $H$ is bounded, we have immediately
 \begin{align} \label{eq:Tn1}
\E  |T^1_n| &\leq   \frac{C }{u_n^{\alpha_0/2}\Delta_n^{1/2 - \alpha_0/4} \sqrt{n}} \sum_{i=0}^{\lfloor \frac{n}{2} \rfloor - 1} \E  \int_{2i \Delta_n}^{2(i+1) \Delta_n}  \int \frac{G^2_s}{\epsilon_n}  \overline{N}(ds,dz), \nonumber \\
& \leq \frac{C }{\epsilon_n u_n^{\alpha_0/2}\Delta_n^{1/2 - \alpha_0/4} \sqrt{n}} \E \int_0^1 \int G_s^2  \overline{N}(ds,dz) \leq  \frac{C }{\epsilon_n u_n^{\alpha_0/2}\Delta_n^{1/2 - \alpha_0/4} \sqrt{n}} .
\end{align}
With $\epsilon_n= \Delta_n^{\alpha_0/8}$, we  deduce $\E |T_n^1 |\to 0$.
 
We now turn to  $T_n^2$.  We first apply Lemma \ref{L:expan} to $f_2$ to obtain
\begin{align*}
\E_{2i}  \int_{2i \Delta_n}^{2(i+1) \Delta_n}  \int H^2_s  \overline{N}(ds,dz)  \leq  & \E_{2i}  (f_2 ( \overline{\rho}_{2i}^n) 
- P_n f_2 (X_{2i \Delta_{n}}, \theta_0))^2
\leq  C u_n^{\alpha_0} \Delta_n^{1- \alpha_0/2} (1+ X_{2i \Delta_n}).
\end{align*}
 From Cauchy-Schwarz inequality it yields
 \begin{align*}
|\E_{2i} & \int_{2i \Delta_n}^{2(i+1) \Delta_n}  \int H_s G_s {\bf 1}_{| G_s| \leq \epsilon_n}  \overline{N}(ds,dz)| \leq \\
& C u_n^{\alpha_0/2} \Delta_n^{1/2- \alpha_0/4} (1+ \sqrt{X_{2i \Delta_n}})\E^{1/2}_{2i}  \int_{2i \Delta_n}^{2(i+1) \Delta_n}  \int G^2_s {\bf 1}_{| G_s| \leq \epsilon_n}  \overline{N}(ds,dz), 
\end{align*}
 and applying once again Cauchy-Schwarz inequality
\begin{align*}
\E |T_n^2|  \leq  
%\frac{\Delta_n }{u_n^{\alpha_0/2}\Delta_n^{1 - \alpha_0/4}} C u_n^{\alpha_0/2} \Delta_n^{1/2- \alpha_0/4} 
 %\left(\E \sum_{i=0}^{\lfloor \frac{1}{2 \Delta_n} \rfloor - 1} (1+ X_{2i \Delta_n}) \right. \\
% & \left.  \times\E \sum_{i=0}^{\lfloor \frac{1}{2 \Delta_n} \rfloor - 1} \int_{2i \Delta_n}^{2(i+1) \Delta_n}  \int G^2_s 1_{| G_s| \leq \epsilon_n}  \overline{N}(ds,dz) 
 %\right)^{1/ 2},\\
  C  \left(\E \frac{1}{n} \sum_{i=0}^{\lfloor \frac{n}{2} \rfloor - 1} (1+ X_{2i \Delta_n}) 
%  \right. \\
% & \left.  
 \times\E \sum_{i=0}^{\lfloor \frac{n}{2} \rfloor - 1} \int_{2i \Delta_n}^{2(i+1) \Delta_n}  \int G^2_s {\bf 1}_{| G_s| \leq \epsilon_n}  \overline{N}(ds,dz) \right)^{1/ 2}.
\end{align*}
The first expectation is bounded from Proposition \ref{Th:Moments}, consequently
\begin{align} \label{eq:Tn2}
\E |T_n^2|  \leq C \E^{1/2} \int_0^1  \int G^2_s {\bf 1}_{| G_s| \leq \epsilon_n}  \overline{N}(ds,dz).
\end{align}
As $\epsilon_n$ goes to zero, we conclude $\E |T_n^2|  \to 0$ by dominated convergence and the proof of (ii) is finished.
\end{proof}

%%%%%
\noindent
2. \underline{ Case $n \Delta_n \to + \infty$.}
The proof follows essentially the same lines as in the previous case and we only indicate the main changes. 
The bounds \eqref{E:momneg1} and \eqref{E:momneg2} still hold.
From the ergodic Theorem (if $b_0>0$, the process $(X_t)$ is geometrically ergodic, see \cite{Li-Ma}), we have the convergence in probability 
for $0<p \leq 1$
\[
\frac{1}{n  \Delta_n} \int_0^{n \Delta_n}
 X_{t}^p dt\xrightarrow[n \to \infty]{} \int_0^{+ \infty} x^p \pi_{0}(dx),
\]
since $\int_0^{+ \infty} x \pi_{0}(dx) < + \infty$.  Moreover we know that $\E(X_t)  \rightarrow a_0/b_0$ (see \cite{BarczyTransformeeLaplace}), then $\sup_t \E(X_t) < + \infty$ and combining this with Proposition \ref{Th:Moments} (iv) we can prove
$$
\E |\frac{1}{n } \sum_{i=0}^{\lfloor \frac{n}{2} \rfloor - 1} X_{2i \Delta_n}^p -\frac{1}{2n  \Delta_n} \int_0^{n \Delta_n}
 X_{t}^p dt |
 \xrightarrow[n \to \infty]{} 0.
$$
We deduce the convergence in probability 
\begin{equation} \label{E:ergo}
\frac{1}{n } \sum_{i=0}^{\lfloor \frac{n}{2} \rfloor - 1} X_{2i \Delta_n}^p  \xrightarrow[n \to \infty]{} \frac{1}{2}\int_0^{+ \infty} x^p \pi_{0}(dx).
\end{equation}
Now replacing $\rho_{2i}^{n}$ by $\overline{\rho}_{2i}^{n}$, we obtain \eqref{E:app} under the additional assumption $\sqrt{n} \Delta_n^{1/ \alpha + \alpha/4 -1/2-\epsilon_0} \rightarrow 0$. It remains to prove the convergence in distribution of 
$$
A_n(\theta_0) \frac{1}{n} \sum_{i=0}^{\lfloor \frac{n}{2} \rfloor - 1} [ f( \overline{\rho}_{2i}^n) 
- P_n f(X_{2i \Delta_n}, \theta_0) ].
$$
We check  \eqref{eq:point1cvTCL}, \eqref{eq:point2cvTCL}, \eqref{eq:point3cvTCL} (with $\overline{\Sigma}$) by proceeding as in the case $n \Delta_n$ fixed but using 
\eqref{E:ergo} and $\sup_t \E(X_t) < + \infty$ to conclude.

%As the asymptotic variance $\overline{\Sigma}(\theta_0)$ is deterministic, we only need the following convergences in probability

%\begin{align}
%    & \sum_{i=0}^{\lfloor \frac{n}{2} \rfloor - 1}  \E_{2i} (\zeta_{2i}^{ n} ) \to 0  \label{eq:TLpoint1cvTCL} \\
  %  & \sum_{i=0}^{\lfloor \frac{n}{2} \rfloor - 1}   \E_{2i} (|\zeta_{2i}^{k, n}|^{p}) \to 0   \quad \text{for some} \; p>2,  \quad k=1,2,3, \label{eq:TLpoint2cvTCL} \\
%    & \sum_{i=0}^{\lfloor \frac{n}{2} \rfloor - 1}   \E_{2i} (\zeta_{2i}^{k, n} \zeta_{2i}^{l, n}) \to \Sigma_{k,l}(\theta_0) 
 %    \quad \text{for} \quad k,l=1,2,3. \label{eq:TLpoint3cvTCL}
%\end{align}
%The convergence \eqref{eq:TLpoint1cvTCL} is immediate. We deduce \eqref{eq:TLpoint2cvTCL} and \eqref{eq:TLpoint3cvTCL} by a similar proof to the case $n \Delta_n=1$.

\subsubsection{Proof of Proposition \ref{Th:LFGN}}
By definition of $F_n(\theta)$, we have 
$$
A_n(\theta)\nabla_{\theta} F_n(\theta)  \Lambda_n(\theta)=- \frac{1}{n} \sum_{i=0}^{\lfloor n/2 \rfloor - 1}A_n(\theta) \nabla_{\theta}P_nf (X_{2i \Delta_n}, \theta) \Lambda_n(\theta).
$$
We start by computing 
$$
\nabla_{\theta} P_n f(x, \theta)= \left(
\begin{array}{ccc}
 \partial_{\sigma^2} P_n f_1 (x, \theta)& \partial_{\delta} P_n f_1(x, \theta)&  \partial_{\alpha} P_n f_1(x, \theta)\\
  \partial_{\sigma^2} P_n f_2 (x, \theta)& \partial_{\delta} P_n f_2(x, \theta)&  \partial_{\alpha} P_n f_2(x, \theta)\\
   \partial_{\sigma^2} P_n f_3 (x, \theta)& \partial_{\delta} P_n f_3(x, \theta)&  \partial_{\alpha} P_n f_3(x, \theta)
\end{array}
\right).
$$
As $P_n f_1(x, \theta) = \exp(-u_n^2 \sigma^2 -2 \delta x^{1-\alpha/2} u_n^{\alpha} \Delta_n^{1 - \alpha/2})$, we have
\begin{align} \label{eq:gradientcos}
    \partial_{\sigma^2} P_n f_1(x, \theta) & = -  u_n^2 P_n f_1(x, \theta), \nonumber \\
    \partial_{\delta} P_n f_1(x, \theta) & = -2 x^{1-\alpha/2} u_n^{\alpha} \Delta_n^{1 - \alpha/2}P_n f_1(x, \theta), \\
    \partial_{\alpha} P_n f_1(x, \theta) & = -2 \delta x^{1-\alpha/2} u_n^{\alpha} \Delta_n^{1 - \alpha/2} \ln( u_n/ \sqrt{x \Delta_n})P_n f_1(x, \theta). \nonumber
 \end{align}
 Next observing that  $\nabla_{\theta} P_n f_2(x, \theta) =- \nabla_{\theta} P_n K(x, \theta)$ (the same connection holds for $f_3$ replacing $K(x)$ by $K(2x)$), we evaluate $\nabla_{\theta} P_n K(x, \theta)$. We recall that from \eqref{eq:defPnf}
 $$
 P_n K(x, \theta)= \E  K(\sqrt{2} u_n \sigma B_1 + \frac{(2 \delta x)^{1/\alpha}}{\sqrt{x}} u_n \Delta_n^{1/\alpha - 1/2} S^{\alpha}_1),
 $$
 and we denote by $g_{n,\theta, x}$ the density of 
 $\sqrt{2} u_n \sigma B_1 + \frac{(2 \delta x)^{1/\alpha}}{\sqrt{x}} u_n \Delta_n^{1/\alpha - 1/2} S^{\alpha}_1$.
 We obtain the following result where $\F f$ is the Fourier transform of $f$, defined by $ \F f(y) = \int_{\R} e^{-iyz}f(z)dz$. 
%%%%%%%%%%
 \begin{lem} \label{L:gradientK}
 Let $K$ be defined by \eqref{eq:defK}. Then we have
 \begin{align*}
  \partial_{\sigma^2}  P_n K (x, \theta) & = u_n^2 P_n K^{\prime \prime} (x, \theta), \\
  \partial_{\delta}  P_n K(x, \theta)  &= - x^{1-\alpha/2} u_n^{\alpha} \Delta_n^{1 - \alpha/2} \frac{2}{2 \pi} \int \F K(y) |y|^{\alpha}\F g_{n,\theta, x}(y) dy, \\
  \partial_{\alpha}  P_n K(x, \theta) & = - \delta x^{1-\alpha/2} u_n^{\alpha} \Delta_n^{1 - \alpha/2} \left(
  \frac{2}{2 \pi} \int \F K(y) |y|^{\alpha} \ln(|y|)\F g_{n,\theta, x}(y) dy \right. \\
& \left. +(\ln(\frac{u_n}{\sqrt{\Delta_n}})   -\ln(\sqrt{x}))\frac{2}{2 \pi} \int \F K(y) |y|^{\alpha}\F g_{n,\theta, x}(y) dy \right).
\end{align*}
 \end{lem} 
 \begin{proof}
 Since $K$ is compactly supported and even, its Fourier transform is a Schwartz function and $\F \F K = 2 \pi K$. We also have $-y^2 \F K (y) = \F K^{\prime \prime} (y)$. With the previous notation it yields 
\begin{align*}
    \nabla_{\theta} P_n K (x, \theta) 
    & =  \nabla_{\theta} \int_{\R} K(y) g_{n,\theta, x}(y) dy 
    = \frac{1}{2\pi} \nabla_{\theta} \int_{\R} ( \F \F K) (y) g_{n,\theta, x}(y) dy \\
    = & \frac{1}{2\pi} \nabla_{\theta} \int_{\R} \F K (y) \F g_{n,\theta, x}(y) dy 
    =  \frac{1}{2\pi} \int_{\R} \F K (y) \nabla_{\theta} \F g_{n,\theta, x}(y) dy,
\end{align*}
where we successively used duality and Lebesgue's Theorem. By definition of $g_{n,\theta, x}$, we have
\begin{equation} \label{eq:fouriergn}
\F g_{n,\theta, x}(y)= \exp(-u_n^2 \sigma^2y^2 -2 |y|^{\alpha} \delta x^{1-\alpha/2} u_n^{\alpha} \Delta_n^{1 - \alpha/2}).
\end{equation}
The computation of $\nabla_{\theta} \F g_{n,\theta, x}$ is straightforward and the expression of $\nabla_{\theta} P_n K $ follows. 
\end{proof}
 Coming back to $\nabla_{\theta} P_n f_2(x, \theta)$ we obtain the following result.
 \begin{lem} \label{L:gradientf}
We assume that $u_n [\ln(1/ \Delta_n)]^{1/2} \to 0$. Then we have $\forall x>0$
 \begin{align*}
 (i) &\;  \left|  \frac{1}{u_n^{\alpha/2} \Delta_n^{1/2- \alpha/4}}  \partial_{\sigma^2}  P_n f_2(x, \theta) \right| \leq  \varepsilon_n^1(x, \theta), \\
(ii)& \; \frac{1}{u_n^{\alpha} \Delta_n^{1- \alpha/2}}  \partial_{\delta}  P_n f_2(x, \theta) =  x^{1- \alpha/2} \psi(\alpha) + x^{1- \alpha/2} \varepsilon_n^2(x, \theta), \\
(iii) & \frac{1}{u_n^{\alpha} \Delta_n^{1- \alpha/2}}  \left(\partial_{\alpha} P_n f_2(x, \theta) -  \delta \ln(\frac{u_n}{\sqrt{\Delta_n}})\partial_{\delta} P_n f_2(x, \theta)\right) 
=  \delta x^{1- \alpha/2} [ \partial_{\alpha} \psi(\alpha)-\ln(\sqrt{x}) \psi(\alpha)] \\
& \quad \quad  +x^{1- \alpha/2}  \varepsilon_n^3(x, \theta)-  x^{1- \alpha/2} \ln(\sqrt{x}) \varepsilon_n^2(x, \theta),
  \end{align*}
 where $\psi(\alpha)=c_{\alpha} \int \frac{f_2(y)}{|y|^{\alpha+1}} dy=c_{\alpha} \int \frac{1-K(y)}{|y|^{\alpha+1}} dy$ and for any compact subset $A \subset \Theta$ 
  $$
 \sup_{\theta \in A} ( |\varepsilon_n^1(x, \theta)|+|\varepsilon_n^2(x, \theta)|+|\varepsilon_n^3(x, \theta)|) \leq C u_n (1+ \sqrt{x}).
  $$
 The same result holds for $f_3$ replacing $\psi(\alpha)$ by $2^{\alpha} \psi(\alpha)$.
  \end{lem}

\begin{proof}
We first prove that
\begin{equation} \label{eq:Fourier}
\frac{2}{2 \pi} \int \F K(y) |y|^{\alpha} dy =c_{\alpha} \int \frac{(1-K(y))}{|y|^{\alpha+1}} dy= \psi(\alpha).
\end{equation}
 By definition of $c_{\alpha}$ we have
 $
 2 |y|^{\alpha}= c_{\alpha}\int_{\R} \frac{1- \cos(uy) }{|u|^{\alpha+1} } du.
 $
 Then for $\epsilon >0$
 \begin{align*}
 \frac{2}{2 \pi} \int \F K(y) |y|^{\alpha} dy &=  \frac{c_{\alpha}}{2 \pi} \int \F K(y)  \left(\int_{\R} \frac{1- \cos(uy) }{|u|^{\alpha+1} } du\right)dy \\
  & =  \frac{c_{\alpha}}{2 \pi} \int \F K(y)  \left(\int_{\R} \frac{1- \cos(uy) }{|u|^{\alpha+1} } ({\bf 1}_{|u| \leq \epsilon} + {\bf 1}_{|u| > \epsilon}) du\right)dy.
 \end{align*}
 Since $1- \cos(uy)=h(uy)u^2y^2$ with $h$ bounded, we can apply Fubini's Theorem
 \begin{align*}
 \frac{2}{2 \pi} \int \F K(y) |y|^{\alpha} dy 
   = & \frac{c_{\alpha}}{2 \pi}   \int_{\R} \frac{u^2 }{|u|^{\alpha+1} } {\bf 1}_{|u| \leq \epsilon}  \left( \int \F K(y)h(uy) y^2dy\right) du \\
& + \frac{c_{\alpha}}{2 \pi}   \int_{\R} \frac{1 }{|u|^{\alpha+1} } {\bf 1}_{|u| >\epsilon}  \left(  \int \F K(y)(1- \cos(uy))dy\right)du.
 \end{align*}
  Moreover, we have
 $$
 \int \F K(y) dy=2 \pi K(0)=2 \pi , \;\ \int \F K(y) \cos(uy) dy=\F(\F K)(y)=2 \pi K(u).
 $$
 Letting $\epsilon$ go to zero, we obtain \eqref{eq:Fourier}.

Next from \eqref{eq:fouriergn}, we have
$$
\sup_{{\theta} \in A} | \F g_{n,\theta, x}(y) -1| \leq C u_n(1+y^2)(1+ \sqrt{x}),
$$
and consequently
\begin{align*}
 \sup_{\theta \in A} |\frac{2}{2 \pi} \int \F K(y) |y|^{\alpha}\F g_{n,\theta, x}(y) dy - \psi(\alpha)| \leq  C u_n (1+ \sqrt{x}),\\
 \sup_{\theta \in A} |\frac{2}{2 \pi} \int \F K(y) |y|^{\alpha} \ln(|y|) \F g_{n,\theta, x}(y) dy - \partial_{\alpha}\psi(\alpha)| \leq C u_n (1+ \sqrt{x}) .
\end{align*}
Since $f_2=1-K$, we obtain (i) combining Lemma \ref{L:gradientK} with Lemma \ref{L:expan} (applied to $K^{\prime \prime}$), (ii) and (iii) follow from 
Lemma  \ref{L:gradientK}, \eqref{eq:Fourier} and the previous uniform bounds.
 \end{proof}
%%%%%%%
With this background, we can finish the proof of Proposition \ref{Th:LFGN}. An explicit calculus leads to $A_n(\theta) \nabla_{\theta}P_nf(x, \theta) \Lambda_n(\theta)=W_n(x, \theta)$ with 
$$
W_n=
 \left(
\begin{array}{ccc}
 \frac{1}{u_n^2}\partial_{\sigma^2} P_n f_1 & \frac{v_n}{u_n^2}\partial_{\delta} P_n f_1 &  \frac{v_n}{u_n^2}[\partial_{\alpha} P_n f_1-\delta \ln(\frac{u_n}{\sqrt{\Delta_n}})\partial_{\delta} P_n f_1] \\
v_n  \partial_{\sigma^2} P_n f_2  & v_n^2\partial_{\delta} P_n f_2 &  v_n^2[\partial_{\alpha} P_n f_2 - \delta \ln(\frac{u_n}{\sqrt{\Delta_n}})\partial_{\delta} P_n f_2  ]\\
  v_n \partial_{\sigma^2} P_n f_3 & v_n^2\partial_{\delta} P_n f_3 &  v_n^2[\partial_{\alpha} P_n f_3- \delta \ln(\frac{u_n}{\sqrt{\Delta_n}})\partial_{\delta} P_n f_3] 
\end{array}
\right), \quad v_n=\frac{1}{u_n^{\alpha/2} \Delta_n^{1/2- \alpha/4}}.
$$
%where $v_n=1/(u_n^{\alpha/2} \Delta_n^{1/2- \alpha/4})$. \\

\noindent
1. \underline{Case $n \Delta_n=1$}.
We  check the uniform convergences in probability for $\alpha \in [\underline{a}, \overline{a}]$, with $1<\underline{a}< \overline{a}<2$
\begin{equation} \label{E:uni1}
\sup_{\alpha \in [\underline{a}, \overline{a}]} |\frac{1}{n} \sum_{i=0}^{\lfloor n/2  \rfloor - 1} X_{2i \Delta_n}^{1- \alpha/2} -\frac{1}{2} \int_0^1 X_t^{1- \alpha/2} dt | \to 0,
\end{equation}
\begin{equation} \label{E:uni2}
\sup_{\alpha \in [\underline{a}, \overline{a}]} |\frac{1}{n} \sum_{i=0}^{\lfloor n/2 \rfloor - 1} X_{2i \Delta_n}^{1- \alpha/2}\ln\sqrt{X_{2 i \Delta_n}} -\frac{1}{2} \int_0^1 X_t^{1- \alpha/2}\ln \sqrt{X_t} dt | \to 0.
\end{equation}
We just detail \eqref{E:uni2}. Let $h(\alpha,x)=x^{1-\alpha/2} \ln (\sqrt{x})$. We have for $\epsilon >0$, 
$$
\forall x >0, \; |h^{\prime}(\alpha,x)| \leq \frac{1}{x^{\alpha/2}} (1 + |\ln(x)|) \leq C(1+ \frac{1}{x^{\alpha/2 + \epsilon}}),
$$
and from Taylor's formula for $x, y >0$
$$
|h(\alpha,y)-h(\alpha,x)| \leq C |y-x| \int_0^1 (1+ \frac{1}{((1-c)x+cy)^{\alpha/2 + \epsilon}}) dc.
$$
Choosing $\epsilon$ such that $\overline{a}/2 + \epsilon <1$, we deduce
$$
\sup_{\alpha \in [\underline{a}, \overline{a}]} |h(\alpha,y)-h(\alpha,x)| \leq C |y-x| (1+ \frac{1}{x}).
$$
From Proposition \ref{Th:Moments} it yields
$$
\E_{2i} \sup_{t \in [2i \Delta_n,  2(i+1) \Delta_n]}  \sup_{\alpha \in [\underline{a}, \overline{a}]} |h(\alpha,X_t)-h(\alpha,X_{2i \Delta_n})| 
\leq C \sqrt{\Delta_n} (1+ X_{2i \Delta_n} + \frac{1}{ X_{2i \Delta_n}}),
$$
and we obtain \eqref{E:uni2}.

Furthermore, from Lebesgue's Theorem $\alpha \mapsto I(\alpha, X)= \int_0^1 X_t^{1- \alpha/2} dt$ admits a derivative and $\partial_{\alpha} I(\alpha,X)=-\int_0^1 X_t^{1-\alpha/2}\ln(\sqrt{X_t}) dt $, so in the expression of $W(\theta)$ \eqref{E:wtheta}, we have
$$
\partial_{\alpha}(\psi(\alpha)I(\alpha,X))=\partial_{\alpha}\psi(\alpha)\int_0^1 X_t^{1-\alpha/2} dt-\psi(\alpha) \int_0^1 X_t^{1-\alpha/2}\ln(\sqrt{X_t}) dt.  
$$
Consequently, combining \eqref{eq:gradientcos} and Lemma \ref{L:gradientf} with \eqref{E:uni1} and \eqref{E:uni2}, we  obtain
$$
\sup_{\theta \in A} |\frac{1}{n} \sum_{i=0}^{\lfloor n/2 \rfloor - 1} W_n(X_{2i \Delta_n}, \theta) + W(\theta)| \to 0,
$$
and we finally conclude from $A_n(\theta)\nabla_{\theta} F_n(\theta) \Lambda_n(\theta)=- \frac{1}{n} \sum_{i=0}^{\lfloor n/2 \rfloor - 1} W_n(X_{2i \Delta_n}, \theta) $. We see easily that $W$ is non-singular since
$$
\det W( \theta)=\frac{1}{8} \delta \psi(\alpha)^2 I(\alpha,X)^2 2^{\alpha} \ln (2) \neq 0. 
$$

\noindent
2. \underline{Case $n \Delta_n \rightarrow \infty$}.
When $n \Delta_n \to + \infty$, we first prove the convergences
\begin{equation} \label{E:uni1ergo}
\sup_{\alpha \in [\underline{a}, \overline{a}]} |\frac{1}{n} \sum_{i=0}^{\lfloor n/2  \rfloor - 1} X_{2i \Delta_n}^{1- \alpha/2} -\frac{1}{2} \int_0^{+ \infty} x^{1- \alpha/2} \pi_{0}(dx) | \to 0,
\end{equation}
\begin{equation} \label{E:uni2ergo}
\sup_{\alpha \in [\underline{a}, \overline{a}]} |\frac{1}{n} \sum_{i=0}^{\lfloor n/2 \rfloor - 1} X_{2i \Delta_n}^{1- \alpha/2}\ln\sqrt{X_{2 i \Delta_n}} -\frac{1}{2} \int_0^{+ \infty} x^{1- \alpha/2}\ln \sqrt{x} \pi_{0}(dx) | \to 0.
\end{equation}
We give the details for $\eqref{E:uni2ergo}$ and study separately  the uniform convergences in probability
\begin{equation} \label{E:uni2ergoa}
\sup_{\alpha \in [\underline{a}, \overline{a}]} |\frac{1}{n} \sum_{i=0}^{\lfloor n/2 \rfloor - 1} X_{2i \Delta_n}^{1- \alpha/2}\ln\sqrt{X_{2 i \Delta_n}} -\frac{1}{2n \Delta_n} \int_0^{n \Delta_n} X_t^{1- \alpha/2}\ln \sqrt{X_t} dt | \to 0,
\end{equation}
\begin{equation} \label{E:uni2ergob}
\sup_{\alpha \in [\underline{a}, \overline{a}]} | \frac{1}{n \Delta_n} \int_0^{n \Delta_n} X_t^{1- \alpha/2}\ln \sqrt{X_t} dt -\int_0^{+ \infty} x^{1- \alpha/2}\ln \sqrt{x} \pi_{0}(dx)| \to 0.
\end{equation}
The proof of \eqref{E:uni2ergoa} is similar to the one of \eqref{E:uni2} and we omit it. It remains to check \eqref{E:uni2ergob}. We set 
$H_n(\alpha)=\frac{1}{n \Delta_n} \int_0^{n \Delta_n} X_t^{1- \alpha/2}\ln \sqrt{X_t} dt$ and $H(\alpha)=\int_0^{+ \infty} x^{1- \alpha/2}\ln \sqrt{x} \pi_{0}(dx)$. The map $\alpha \to H_n(\alpha)$ is continuous on $[\underline{a}, \overline{a}]$ and $H_n(\alpha) \to H(\alpha)$, $ \forall \alpha \in [\underline{a}, \overline{a}]$. Moreover  for $\alpha_1, \alpha_2 \in [\underline{a}, \overline{a}]$ and $x>0$, $\exists c \in (\alpha_1, \alpha_2 )$ such that
\begin{align*}
| x^{1-\alpha_1/2} \ln(\sqrt{x})-x^{1-\alpha_2/2} \ln(\sqrt{x})|  \leq |\alpha_1- \alpha_2| (\ln(x))^2 x^{1-c/2} 
 \leq C |\alpha_1- \alpha_2| (1+ x).
\end{align*}
With $\sup_t \E X_t < + \infty$, it follows that $\forall \delta>0$, 
$
\sup_n \E \sup_{ |\alpha_1- \alpha_2| \leq \delta} | H_n(\alpha_1)-H_n(\alpha_2)| \leq C \delta,
$ 
so the sequence $(H_n(\alpha)_{\alpha \in [\underline{a}, \overline{a}]})$ is tight and we obtain the uniform convergence \eqref{E:uni2ergob}.

Observing that  $\partial_{\alpha} \overline{I}(\alpha)=-\int_0^{\infty} x^{1-\alpha/2} \ln(\sqrt{x} ) \pi_{0}(dx)$ and using the same arguments as in the case $n \Delta_n$ fixed ,we obtain from \eqref{eq:gradientcos},  Lemma \ref{L:gradientf}, \eqref{E:uni1ergo} and \eqref{E:uni2ergo}
$$
\sup_{\theta \in A} |\frac{1}{n} \sum_{i=0}^{\lfloor n/2 \rfloor - 1} W_n(X_{2i \Delta_n}, \theta) + \overline{W}(\theta)| \to 0.
$$
This achieves the proof of Proposition \ref{Th:LFGN}. 
 
\subsection{Main results}

\subsubsection{Proof of Theorem \ref{Th:Existence} and Theorem \ref{Th:ExistenceTempslong} }

In both cases, $n \Delta_n$ fixed and $n \Delta_n \rightarrow + \infty$, we apply Theorem A.2. in Mies and Podolskij \cite{MiesPodolskij}, that extends the results of Jacod and S{\o}rensen \cite{JacodSorensenUniforme}. We just detail the first case, where we have strengthened condition (E.1) in \cite{MiesPodolskij} by assuming the stable convergence in law.  With our notation, the result of Theorem   \ref{Th:Existence} is a consequence of the following conditions:

(E.1) $A_n(\theta_0) F_n(\theta_0) \xRightarrow[n \to \infty]{\mathcal{L}-s} \Sigma(\theta_0)^{1/2} \mathcal{N}$,

(E.2) $\theta \to F_n(\theta)$ is $\mathcal{C}^1$ and for $r_n$ a sequence of real numbers we have the convergences in probability 
\[ \sup_{\theta \in \mathcal{B}_{r_n}(\theta_0) } 
|| A_n(\theta) \nabla_\theta F_n(\theta) \Lambda_n(\theta) - W(\theta) || \xrightarrow[n \to \infty]{} 0, \;
 \sup_{\theta \in \mathcal{B}_{r_n}(\theta_0) }  \frac{||\Lambda_n(\theta)|| \; ||A_n(\theta) A_n(\theta_0)^{-1}||}{r_n}  \xrightarrow[n \to \infty]{} 0 ,\]
%(E.3) 
\begin{align*} \text{(E.3)} \; \sup_{\theta \in \mathcal{B}_{r_n}(\theta_0) } 
(|| A_n(\theta) A_n(\theta_0)^{-1} - Id || +  
 || \Lambda_n(\theta) \Lambda_n(\theta_0)^{-1} - Id || + || W(\theta) - W(\theta_0)||) \xrightarrow[n \to \infty]{} 0,
\end{align*}
where $\mathcal{B}_{r_n}(\theta_0)$ is the ball with center $\theta_0$ and radius $r_n$, $|| \cdot ||$ a matrix norm and $Id$ the identity matrix.
Condition (E.1) is proven in Proposition \ref{Th:TCL} and the first condition of (E.2) is proven in Proposition \ref{Th:LFGN}. The second condition in (E.2) and condition (E.3) are easily checked with $r_n = 1/\ln(\Delta_n)^2$.

\subsubsection{Proof of Theorem \ref{Th:Unicité}}
The existence is a consequence of Theorem \ref{Th:Existence}. To prove the global uniqueness result we consider the renormalized estimating function (restricted to $\beta=(\sigma^2, \delta)$ with $f=(f_k)_{1 \leq k \leq2}$)
$$
\tilde{F}_n(\sigma^2, \delta) = \tilde{A}_n(\alpha_0) F_n( \sigma^2, \delta, \alpha_0), \text{with} \;
\tilde{A}_n(\alpha_0)=diag\left(\frac{1}{u_n^2}, \frac{1}{u_n^{\alpha_0} \Delta_n^{1 - \alpha_0/2}}\right).
% \tilde{A}_n(\alpha_0)=\left(\begin{array}{cc} \frac{1}{u_n^2} & 0 \\ 0 & \frac{1}{u_n^{\alpha_0} \Delta_n^{1 - \alpha_0/2}} \end{array}\right). 
$$
From Proposition \ref{Th:TCL}, we deduce $\tilde{F}_n(\sigma_0^2, \delta_0) \to 0$. Next, an explicit calculus gives 
\[ \nabla_{\beta} \tilde{F}_n(\sigma^2, \delta) =- \frac{1}{n} \sum_{i=0}^{\lfloor n/2 \rfloor - 1} 
\left(\begin{array}{cc}
\frac{1}{u_n^2} \partial_{\sigma^2} P_n f_1 
& \frac{1}{u_n^2} \partial_{\delta} P_n f_1 \\
\frac{1}{u_n^{\alpha_0} \Delta_n^{1 - \alpha_0/2}} \partial_{\sigma^2} P_n f_2
& \frac{1}{u_n^{\alpha_0} \Delta_n^{1 - \alpha_0/2}} \partial_{\delta} P_n f_2
\end{array}\right),\]
and from Proposition \ref{Th:LFGN}, we have the  convergence in probability for any compact subset A of $(0, + \infty) \times (0, + \infty)$
   \[\sup_{(\sigma^2, \delta) \in A} \left|\left| \nabla \tilde{F}_n(\sigma^2, \delta) - \nabla \tilde{F}(\sigma^2, \delta) \right|\right| \to 0,
   \; \text{where} \;
   \tilde{F}(\sigma^2, \delta) = \left(\begin{array}{c}
\frac{1}{2} (\sigma^2 - \sigma_0^2) \\
\frac{1}{2} \psi(\alpha_0) \int_0^1 X_t^{1-\alpha_0/2} dt (\delta_0 - \delta)
\end{array}\right).
    \]
 The global uniqueness result follows then from Theorem 2.7.a) of Jacod and S{\o}rensen \cite{JacodSorensenUniforme}, since $(\sigma_0^2, \delta_0)$ is the unique root of $\tilde{F}(\sigma^2, \delta)=0$. 

\vspace{0.5cm}
\noindent
{\bf Funding. This work was supported by french ANR reference ANR-21-CE40-0021.}

\bibliographystyle{plain}

%\bibliography{biblio}

\end{document}